\documentclass{siamltex}
\usepackage{amsmath}
\usepackage{amsfonts}
\usepackage{amssymb}
\usepackage{graphicx}
\usepackage{color}
\usepackage{bm}
\usepackage{verbatim}
\usepackage{hyperref}
\usepackage{bmpsize}
\allowdisplaybreaks

\newtheorem{remark}[theorem]{Remark}

\newcommand{\be}{\begin{equation}}
\newcommand{\ee}{\end{equation}}
\newcommand{\bea}{\begin{eqnarray}}
\newcommand{\eea}{\end{eqnarray}}
\newcommand{\beas}{\begin{eqnarray*}}
\newcommand{\eeas}{\end{eqnarray*}}

\newcommand{\vertiii}[1]{{\left\vert\kern-0.25ex\left\vert\kern-0.25ex\left\vert #1 
    \right\vert\kern-0.25ex\right\vert\kern-0.25ex\right\vert}}

\begin{document}
\title{An efficient algorithm for simulating\\ ensembles of parameterized flow problems}

\author{
Max Gunzburger\thanks{Department of Scientific Computing,  
Florida State University,
Tallahassee, FL 32306-4120 {\tt mgunzburger@fsu.edu}. Research supported by the U.S. Air Force Office of Scientific Research grant FA9550-15-1-0001 and the U.S. Department of Energy Office of Science grants DE-SC0009324 and DE-SC0016591.} 
\and Nan Jiang\thanks{Department of Mathematics and Statistics,  
Missouri University of Science and Technology,
Rolla, MO 65409-0020 {\tt jiangn@mst.edu}. }
\and Zhu Wang\thanks{Department of Mathematics, 
University of South Carolina, Columbia, SC 29208 {\tt wangzhu@math.sc.edu}. Research  supported by the
U.S. National Science Foundation grant DMS-1522672 and the U.S. Department of Energy grant  DE-SC0016540.
}
}
%\date{\today}
\maketitle

\begin{abstract}
Many applications of computational fluid dynamics require multiple simulations of a flow under different input conditions. In this paper, a numerical algorithm is developed to efficiently determine a set of such simulations in which the individually independent members of the set are subject to different viscosity coefficients, initial conditions, and/or body forces. The proposed scheme applied to the flow ensemble leads to need to solve a single linear system with multiple right-hand sides, and thus is computationally more efficient than solving for all the simulations separately. We show that the scheme is nonlinearly and long-term stable under certain conditions on the time-step size and a parameter deviation ratio. Rigorous numerical error estimate shows the scheme is of first-order accuracy in time and optimally accurate in space. Several numerical experiments are presented to illustrate the theoretical results. 
\end{abstract}

\section{Introduction}

Numerical simulations of incompressible viscous flows have important applications in engineering and science. In this paper, we consider settings in which one wishes to obtain solutions for several different values of the physical parameters and several different choices for the forcing functions appearing in the partial differential equation (PDE) model. For example, in building low-dimensional surrogates for the PDE solution such as sparse-grid interpolants or proper orthogonal decomposition approximations, one has to first determine expensive approximation of solutions corresponding to several values of the parameters. Sensitivity analysis of solutions is setting in which one often has to determine approximate solutions for several parameter values and/or forcing functions. An important third example is quantifying the uncertainties of outputs from the model equations. Mathematical models should take into account the uncertainties invariably present in the specification of physical parameters and/or forcing functions appearing in the model equations. 
For flow problems, because the viscosity of the liquid or gas often depends on the temperature, an inaccurate measurement of the temperature would introduce some uncertainty into the viscosity of the flow. Direct measurements of the viscosity using flow meters and measurements of the state of the system are also prone to uncertainties. Of course, forcing functions, e.g., initial condition data, can and usually are also subject to uncertainty.
In such cases, due to the lack of of exact information, stochastic modeling is used to describe flows subject to a random viscosity coefficient and/or random forcing. Subsequently,  numerical methods are employed to quantify the uncertainties in system output.
% (e.g., in \cite{powell2012preconditioning,sousedik2016stochastic}).  
It is known that uncertainty quantification (UQ), when a random sampling method such as Monte Carlo method is used, could be computationally expensive for large-scale problems because each individual realization requires a large-scale computation but on the other hand, many realizations may be needed in order to obtain accurate statistical information about the outputs of interest. Therefore, for all the examples discussed and for many others, how to design efficient algorithms for performing multiple numerical simulations becomes a matter of great interest. 

The ensemble method which forms the basis for our approach was proposed in \cite{JL14}; there, a set of $J$ solutions of the Navier-Stokes equations (NSE) with distinct initial conditions and forcing terms is considered. 
All solutions are found, at each time step, by solving a linear system with one shared coefficient matrix and $J$ right-hand sides (RHS), reducing both the storage requirements and computational costs of the solution process. 
The algorithm of \cite{JL14} is first-order accurate in time; it is extended to higher-order accurate schemes in \cite{J15,J16}. Ensemble regularization methods are developed in \cite{J15, JL15, TNW16} for high Reynolds number flows, and a turbulence model based on ensemble averaging is developed in \cite{JKL15}. The ensemble algorithm has also been extended to simulate MHD flows in \cite{MR16}. Ensemble algorithms incorporating reduced-order modeling techniques are studied in \cite{GJS16a, GJS16b}. It is worth mentioning that all the ensemble algorithms developed so far can only deal with simulations subject to different initial conditions and/or body forces, but not other model parameters.

In this paper, we develop a numerical scheme for ensemble-based simulations of the NSE in which not only the initial data and body force function, but also the {\it viscosity coefficient}, may vary from one ensemble member to another.   
Specifically, we consider a set of $J$ NSE simulations on a bounded domain subject to no-slip boundary conditions for which, for $j=1,\ldots,J$, an individual member solves the system
\begin{equation}
\label{eq:NSE}
\left\{
\begin{array}{rcll}
u_{j,t}+u_{j}\cdot\nabla u_{j}-\nu_j\triangle u_{j}+\nabla p_{j}  &=& f_{j}(x,t) \quad &\text{ in }\Omega\times [0, \infty) \\
\nabla\cdot u_{j}  &=& 0 \quad &\text{ in }\Omega \times [0, \infty) \\
u_{j}  &=& 0 \quad &\text{ on }\partial\Omega\\
u_{j}(x,0)  &=& u_{j}^{0}(x) \quad & \text{ in }\Omega 
\end{array}\right.,
\end{equation}
which corresponds, for each $j$, to a different viscosity coefficient $\nu_j$ and/or distinct initial data $u_{j}^{0}$ and/or body forces $f_{j}$.

Due to the nonlinear convection term, implicit and semi-implicit schemes are invariably used for time integration.  For a semi-implicit scheme, the associated discrete linear systems would be different for each individually independent simulation, i.e., for each $j$.  As a result, at each time step, $J$ linear systems need to be solved to determine the ensemble, resulting in a huge computational effort. For a fully implicit scheme, the situation is even worse because one would have to solve many more linear systems due to the nonlinear solver iteration. To tackle this issue, we propose a novel discretization scheme that results, at each time step, in a common coefficient matrix for all the ensemble members. 
%which keeps the coefficient matrix independent of the ensemble member, hence, only one linear system needs to be solved at each time step

\subsection{The ensemble-based semi-implicit scheme}

For clarity, we temporarily suppress the spatial discretization and only consider the ensemble-based implicit-explicit temporal integration scheme 
\begin{equation}
\label{First-Order}
\left\{\begin{aligned}
\frac{u_{j}^{n+1}-u_{j}^{n}}{\Delta t}+\overline{u}^{n}\cdot\nabla u_{j}^{n+1}
+(u_{j}^{n}-\overline{u}^{n})\cdot\nabla u_{j}^{n} +\nabla p_{j}^{n+1}\qquad\\
-\overline{\nu}\Delta u_{j}^{n+1}-\left(\nu_j-\overline{\nu}\right)\Delta u_{j}^{n}&=f_{j}^{n+1}\\
\nabla\cdot u_{j}^{n+1}&=0,
\end{aligned}\right.
\end{equation}
where $\overline{u}^{n}$ and $\overline{\nu}$ are the ensemble means of the velocity and viscosity coefficient, respectively, defined as
\[
\overline{u}^{n}:=\frac{1}{J}\sum_{j=1}^{J}u_{j}^{n} \qquad \text{and}\qquad \overline{\nu}:=\frac{1}{J}\sum_{j=1}^{J}\nu_{j}.
\]
After rearranging the system, we have, at time $t_{n+1}$,  
\begin{equation}
\label{First-Order-2}
\left\{\begin{aligned}
\frac{1}{\Delta t}u_{j}^{n+1}+\overline{u}^{n}\cdot\nabla u_{j}^{n+1}
&-\overline{\nu}\Delta u_{j}^{n+1}+\nabla p_{j}^{n+1}\\
&= f_{j}^{n+1}+ \frac{1}{\Delta t}u_{j}^{n} -(u_{j}^{n}-\overline{u}^{n})\cdot\nabla u_{j}^{n}+\left(\nu_j-\overline{\nu}\right)\Delta u_{j}^{n} \\
\nabla\cdot u_{j}^{n+1}&=0.
\end{aligned}\right.
\end{equation}
It is clear that the coefficient matrix of the resulting linear system will be independent of $j$. 
Thus, for the flow ensemble, to advance all members of the ensemble one time step, we need only solve a single linear system with $J$ right-hand sides. Compared with solving $J$ individually independent simulations, this approach used with a block solver such as a block generalized CG method \cite{FOP95,O80} is much more efficient and significantly reduces the required storage. 
When the size of the ensemble becomes huge, it can be subdivided into $p$ sub-ensembles so as to balance memory, communication, and computational costs and then \eqref{First-Order} can be applied to each sub-ensemble.

The rest of this section is devoted to establishing notation and to providing other preliminary information. Then, in \S\ref{sec:stab}, we prove a conditional stability result for a fully discrete finite element discretization of \eqref{First-Order}. In \S\ref{sec:err}, we derive an error estimate for the fully-discrete approximation. Results of the preliminary numerical simulations that illustrate the theoretical results are given in \S\ref{sec:num} and \S\ref{sec:con} provides some concluding remarks.

\subsection{Notation and preliminaries}

Let $\Omega$ denote an open, regular domain in $\mathbb{R}^{d}$ for $d=2$ or 
$3$ having boundary denoted by $\partial\Omega$. The $L^{2}(\Omega)$ norm and inner product are denoted by $\|\cdot\|$ and
$(\cdot, \cdot)$, respectively. The $L^{p}(\Omega)$ norms and the Sobolev
$W^{k}_{p}(\Omega)$ norms are denoted by $\|\cdot\|_{L^{p}}$ and $\|\cdot\|_{W_{p}^{k}}$,
respectively. 
The Sobolev space $W_{2}^{k}(\Omega)$ is simply denoted by $H^{k}(\Omega)$ and its norm by $\|\cdot\|_{k}$. For functions $v(x,t)$ defined on $(0,T)$, we define, for
$1\leq m<\infty$,
\[
\| v \|_{\infty,k} \text{ }:=EssSup_{[0,T]}\| v(\cdot, t)\|_{k}\qquad \text{and}\qquad \|v\|_{m,k} \text{ }:= \Big(  \int_{0}^{T}\|v(\cdot, t)\|_{k}^{m}\, dt\Big)
^{1/m} \text{ .}%
\]
Given a time step $\Delta t$, associated discrete norms are defined as 
\[
\vertiii{v}_{\infty,k}=\max\limits_{0\leq n\leq N}\Vert v^{n}\Vert_{k} \qquad \text{and}\qquad
\vertiii{v}_{m,k}:= \Big(\sum_{n=0}^{N}||v^{n}||_{k}^{m}\Delta t\Big)^{1/m},
\]
where $v^n=v(t_n)$ and $t_n=n\Delta t$.
%
%For $v\in H^1(\Omega)$, the $H^{1/2}(\Omega)$ norm satisfies the interpolation
%inequality
%\[
%\Vert v\Vert_{1/2}\leq C\sqrt{\Vert v\Vert\Vert\nabla v\Vert}.
%\]
Denote by 
%$H^{-k}(\Omega)$ the dual space of bounded linear functions on $H^{k}_{0}(\Omega)$. A norm of the 
$H^{-1}(\Omega)$ the dual space of bounded linear functions on $H^{1}_{0}(\Omega)=\{v\in H^{1}\,:\, v=0 \,\mbox{on $\partial\Omega$}\}$; a norm on $H^{-1}(\Omega)$ is given by
\[
\|f\|_{-1}=\sup_{0\neq v\in H^{1}_{0}(\Omega)}\frac{(f,v)}{\Vert\nabla v\Vert}
\text{ .}
\]

The velocity space
$X$ and pressure space $Q$ are given by 
\[
X:=[H_{0}^{1}(\Omega)]^{d}  \qquad \text{and}\qquad Q:=L_{0}^{2}(\Omega)=\{q\in L^2(\Omega)\,\,:\,\, \int_\Omega q\, d\Omega =0\},
\]
respectively. The space of weakly divergence free functions is
\[
V\text{ }:=\{v\in X\,\,:\,\, (\nabla\cdot v,q)=0\,\,\forall q\in Q\} .
\]
%and the norm on $V^{\ast}$ (the dual of $V$) is defined as
%\[
%\|f\|_{\ast}=\sup_{0\neq v\in V}\frac{(f,v)}{\Vert\nabla v\Vert} .
%\]

A weak formulation of (\ref{eq:NSE}) reads: for $j=1,\ldots,J$, find $u_j: \,[0,T]\rightarrow X$ and 
$p_j:\, [0,T]\rightarrow Q$ for a.e. $t\in(0,T]$ satisfying
\begin{equation*}
\left\{\begin{aligned}
(u_{j,t},v)+(u_{j}\cdot\nabla u_{j},v)+\nu_j(\nabla u_{j},\nabla v)-(p_{j}%
,\nabla\cdot v)  &= (f_{j},v) &\forall v\in X \\
(\nabla\cdot u_{j},q)  &= 0   &\forall q\in Q
\end{aligned}\right.
\end{equation*}
with $u_{j}(x,0)=u_{j}^{0}(x)$.

Our analysis is based on a finite element method (FEM) for spatial discretization. However, the results also extend, without much difficulty, to other variational discretization methods. 
Let $X_{h}\subset X$ and $Q_{h}\subset Q$ denote families of conforming velocity and pressure finite element spaces on regular subdivision of $\Omega$ into simplicies; the family is parameterized by the maximum diameter $h$ of any of the simplicies. 
Assume that the pair of spaces $(X_h,Q_h)$ satisfy the discrete inf-sup (or $LBB_h$) condition required for the stability of the finite element approximation and that the finite element spaces satisfy the approximation properties  
\begin{align}
\inf_{v_h\in X_h}\| v- v_h \|&\leq C h^{k+1}\Vert u \Vert_{k+1}	   &\forall v\in [H^{k+1}(\Omega)]^d \label{Interp1}\\
\inf_{v_h\in X_h}\| \nabla ( v- v_h )\|&\leq C h^k \Vert v\Vert_{k+1}&\forall v\in [H^{k+1}(\Omega)]^d \label{interp2}\\
\inf_{q_h\in Q_h}\| q- q_h \|&\leq C h^{s+1}\Vert p\Vert_{s+1}	   &\forall q\in H^{s+1}(\Omega),      \label{interp3}
\end{align}
where the generic constant $C>0$ is independent of mesh size $h$.  
An example for which the $LBB_h$ stability condition and the approximation properties are satisfied is the family of Taylor-Hood $P^{s+1}$--$P^{s}$, $s\geq 2$, element pairs.
For details concerning finite element methods see \cite{Cia02} and see
\cite{GR79,GR86,Max89,Layton08} for finite element methods for the Navier-Stokes equations.

The discretely divergence free subspace of $X_{h}$ is defined as
\[
V_{h}\text{ }:=\{v_{h}\in X_{h}\,\,:\,\,(\nabla\cdot v_{h},q_{h})=0\quad\forall
q_{h}\in Q_{h}\}.
\]
Note that, in general, $V_h\not\subset V$.
We assume the mesh and finite element spaces satisfy the standard inverse inequality 
\begin{equation}\label{inverse}
h\Vert\nabla v_{h}\Vert   \leq C_{(inv)}\Vert v_{h}\Vert. 
\qquad\forall v_{h}\in X_{h}
\end{equation}
that is known to hold for standard finite element spaces with locally quasi-uniform meshes \cite{BS08}. 
We also define the standard explicitly skew-symmetric trilinear form
\[
b^{\ast}(u,v,w):=\frac{1}{2}(u\cdot\nabla v,w)-\frac{1}{2}(u\cdot\nabla w,v) 
\]
that satisfies the bounds \cite{Layton08}
\begin{gather}
b^{\ast}(u,v,w)\leq C \left(\Vert \nabla u\Vert\Vert u\Vert\right)^{1/2}\Vert\nabla v\Vert\Vert\nabla
w \Vert \quad \forall\, u, v, w \in X \label{In1}\\
b^{\ast}(u,v,w)\leq C \Vert \nabla u\Vert\Vert\nabla v\Vert\left(\Vert\nabla
w \Vert\Vert w\Vert\right)^{1/2} \quad \forall\, u, v, w \in X .\label{In2}
\end{gather}
%Let $t^{n}=n\Delta t$, $n=0,1, \ldots ,N$ and $T:=N\Delta t$. 
We also denote the exact and approximate solutions at $t=t^n$ as $u_{j}^{n}$ and $u_{j, h}^{n}$, respectively. 

\section{Stability analysis}\label{sec:stab}

The fully-discrete finite element discretization of (\ref{First-Order}) is given as follows. Given $u_{j,h}^{0}\in X_{h}$, for $n=0,1,\ldots,N-1$, find
 $u_{j,h}^{n+1}\in X_{h}$ and $p_{j,h}^{n+1}\in Q_{h}$ satisfying
\begin{equation}
\label{First-Order-h}
\hspace{-3mm}\left\{\begin{aligned}
&\Big(\frac{u_{j,h}^{n+1}-u_{j,h}^{n}}{\Delta t},v_{h}\Big)+b^{\ast}(\overline{u}_{h}^{n},u_{j,h}^{n+1},v_{h})+b^{\ast}(u_{j,h}^{n}-\overline{u}_{h}^{n},u_{j,h}^{n}, v_{h})-(p_{j,h}^{n+1},\nabla\cdot v_{h})\\ 
&\qquad\qquad+\overline{\nu}(\nabla u_{j,h}^{n+1},\nabla
v_{h})+\left(\nu_j-\overline{\nu}\right)(\nabla u_{j,h}^{n},\nabla
v_{h})=(f_{j}^{n+1},v_{h})\quad \forall v_{h}\in X_{h}\\
&\big(\nabla\cdot u_{j,h}^{n+1},q_{h}\big)=0 \quad\forall q_{h}\in Q_{h}.
\end{aligned}\right.
\end{equation}
We begin by proving the conditional, nonlinear, long-time stability of the scheme
(\ref{First-Order-h}) under a time-step condition and a parameter deviation condition. 

\begin{theorem}[Stability] 
\label{th:First-Order}
For all $j= 1, \ldots, J$, if for some $\mu$, $0\leq\mu<1$, and some $\epsilon$, $0< \epsilon\leq 2-2\sqrt{\mu}$, the following time-step condition and parameter deviation condition both hold
\begin{align}
C\frac{\Delta t}{\overline{\nu} h}\left\Vert\nabla(u_{j,h}^{n}-\overline{u}_{h}^{n})\right\Vert^{2}
&\leq \frac{(2-2\sqrt{\mu}-\epsilon)\sqrt{\mu}}{2(\sqrt{\mu}+\epsilon)} ,
 \label{ineq:CFL-h1}
 \\
 \frac{|\nu_j -\overline{\nu}| }{\overline{\nu}} &\leq \sqrt{\mu},
 \label{ineq:CFL-h2}
\end{align}
then,  the scheme \eqref{First-Order-h} is nonlinearly, long time stable. 
In particular, for $j= 1, \ldots, J$ and for any $N\geq1$, we have 
%$0\leq  \beta < \frac{1}{2}$ and
%
\begin{align*}
&\frac{1}{2}\|u_{j,h}^{N}\|^{2}+\frac{1}{4}\sum_{n=0}^{N-1}\|u_{j,h}%
^{n+1}-u_{j,h}^{n}\|^{2}
+\overline{\nu}\Delta t \left(\frac{\sqrt{\mu}}{2}\frac{2+\epsilon}{\sqrt{\mu}+\epsilon}-\frac{\vert \nu_j-\overline{\nu}\vert}{2 \overline{\nu}}\right)
\|\nabla u_{j,h}^{N}\|^{2}\\
&\qquad\leq\sum_{n=0}^{N-1}\frac{\Delta t}{\overline{\nu}}\|f_{j}^{n+1}\|_{-1}^{2}+ \frac{1}%
{2}\|u_{j,h}^{0}\|^{2}
+\overline{\nu}\Delta t \left(\frac{\sqrt{\mu}}{2}\frac{2+\epsilon}{\sqrt{\mu}+\epsilon}-\frac{\vert \nu_j-\overline{\nu}\vert}{2 \overline{\nu}}\right)\|\nabla u_{j,h}^{0}\|^{2}.
\end{align*}
%\begin{align}
%&\frac{1}{2}\|u_{j,h}^{N}\|^{2}+\frac{1}{4}\sum_{n=0}^{N-1}\|u_{j,h}%
%^{n+1}-u_{j,h}^{n}\|^{2}+\overline{\nu}\Delta
%t (\frac{(3-\sqrt{\mu})\sqrt{\mu}}{2}-\frac{\vert \nu_j-\overline{\nu}\vert}{2 \overline{\nu}})\|\nabla u_{j,h}^{N}\|^{2}\label{Stability}\\
%&\leq\sum_{n=0}^{N-1}\frac{\Delta t}{\nu}\|f_{j}^{n+1}\|_{*}^{2}+ \frac{1}%
%{2}\|u_{j,h}^{0}\|^{2}+\overline{\nu}\Delta
%t (\frac{(3-\sqrt{\mu})\sqrt{\mu}}{2}-\frac{\vert \nu_j-\overline{\nu}\vert}{2 \overline{\nu}})\|\nabla u_{j,h}^{0}\|^{2}, \text{  j= 1,...,J
%.}\nonumber
%\end{align}
\end{theorem}
\begin{proof}
The proof is given in Appendix \ref{proofa1}.
\end{proof}

\begin{remark}
It is seen from \eqref{ineq:CFL-h1} that the upper bound in the time-step condition increases as $\epsilon$ decreases. As $\epsilon\rightarrow 0$, the bound approaches  $1-\sqrt{\mu}$. 
Because the upper bound for the relative deviation of viscosity coefficient in \eqref{ineq:CFL-h2} is bounded by $\sqrt{\mu}$, the two stability conditions are oppositional to each other. 
\end{remark}
\begin{remark}
Noting that the condition \eqref{ineq:CFL-h1} only depends on known quantities such as the solution at $t_n$ and that the scheme \eqref{First-Order-h} is a one-step method, 
\eqref{ineq:CFL-h1} can be used to adapt $\triangle t$ in order to guarantee the stability for the ensemble simulations. 
\end{remark}

\section{Error Analysis\label{sec:err}}

In this section, we give a detailed error analysis of the proposed method under the same type of time-step condition (with possibly different constant $C$ on the left hand side of the inequality) and the same parameter deviation condition. Assuming that
$X_{h}$ and $Q_{h}$ satisfy the $LBB^{h}$ condition, the scheme \eqref{First-Order-h}
is equivalent to: Given $u_{j,h}^{0}\in V_{h}$, for $n=0,1,\ldots,N-1$, find $u_{j,h}%
^{n+1}\in V_{h}$ such that 
\begin{equation}
\label{eq: conv}
\begin{aligned}
&\Big(\frac{u_{j,h}^{n+1}-u_{j,h}^{n}}{\Delta t},v_{h}\Big)+b^{\ast}(\overline{u}_{h}^{n},u_{j,h}^{n+1},v_{h})+b^{\ast}(u_{j,h}^{n}-\overline{u}_{h}^{n},u_{j,h}
^{n},v_{h})\\
&\quad+\overline{\nu}(\nabla u_{j,h}^{n+1},\nabla v_{h})+\left(\nu_j-\overline{\nu}\right)(\nabla u_{j,h}^{n+1},\nabla v_{h})=(f_{j}^{n+1},v_{h})\quad\forall
v_{h}\in V_{h}.
\end{aligned}
\end{equation}
To analyze the rate of convergence of the approximation, we assume
that the following regularity for the exact solutions:
\begin{gather*}
u_{j} \in L^{\infty}(0,T;H^{k+1}(\Omega))\cap H^{1}(0,T;H^{k+1}(\Omega))\cap
H^{2}(0,T;L^{2}(\Omega)),\\
p_{j} \in L^{2}(0,T;H^{s+1}(\Omega))\quad \text{and}\quad f_{j} \in L^{2}%
(0,T;L^{2}(\Omega)).
\end{gather*}
Let $e_{j}^{n}=u_{j}^{n}-u_{j,h}^{n}$ denote the approximation error of the $j$-th simulation at the time instance $t_n$. We then have the following error estimates.

\begin{theorem}[Convergence of scheme \eqref{First-Order-h}]\label{th:errBEFE-Ensemble} 
For all $j= 1, \ldots, J$, if for some $\mu$, $0\leq\mu<1$, and some $\epsilon$, $0< \epsilon\leq 2-2\sqrt{\mu}$, the following time-step condition and parameter deviation condition both hold
\begin{align}
C\frac{\Delta t}{\overline{\nu} h}\left\Vert\nabla(u_{j,h}^{n}-\overline{u}_{h}^{n})\right\Vert^{2}
&\leq \frac{(2-2\sqrt{\mu}-\epsilon)\sqrt{\mu}}{2(\sqrt{\mu}+\epsilon)} ,\label{conv1}
 \\
 \frac{|\nu_j -\overline{\nu}| }{\overline{\nu}} &\leq \sqrt{\mu},\label{conv2}
\end{align}
then,  there exists a positive constant $C$ independent of the time step such that 
\begin{equation*}%\label{ineq:err00}
\begin{aligned}
\frac{1}{2}&\Vert e_{j}^{N}\Vert^{2}
+ \left(\frac{\sqrt{\mu}}{2}\frac{(2+\epsilon)}{\sqrt{\mu}+\epsilon}-\frac{\vert \nu_j-\overline{\nu}\vert}{2 \overline{\nu}}\right)\overline{\nu}\Delta t \Vert\nabla e_{j}^{N}\Vert^{2}\\
&\qquad+\frac{1}{15} \frac{\epsilon} {\sqrt{\mu}+\epsilon}( 1-\frac{\sqrt{\mu}}{2})\overline{\nu} \Delta t\sum_{n=0}^{N-1}\Vert\nabla e_{j}^{n+1}\Vert^{2}
\\
&\leq e^{\frac{CT}{\overline{\nu}^{3}}}
\Big\{\frac{1}{2}\Vert e_{j}^{0}\Vert^{2}
+\left(\frac{\sqrt{\mu}}{2}\frac{(2+\epsilon)}{\sqrt{\mu}+\epsilon}-\frac{\vert \nu_j-\overline{\nu}\vert}{2 \overline{\nu}}\right)\overline{\nu}\Delta t \Vert\nabla e_{j}^{0}\Vert^{2}
\\
&\quad+C\Delta t^2\frac{\vert \nu_j -\overline{\nu}\vert^2}{\overline{\nu}}\vertiii{\nabla u_{j,t}}^2_{2,0}+C\overline{\nu}h^{2k}\vertiii{ u_j }^2_{2,k+1}
+C\frac{\vert \nu_j - \overline{\nu}\vert^2}{\overline{\nu}}h^{2k}\vertiii{ u_j}^2_{2,k+1}
\\
&\quad+Ch^{2k+1}\Delta t^{-1}\vertiii{ u_j}^2_{2,k+1}+C h \Delta t \vertiii{ \nabla u_{j,t}}^2_{2,0} + C\overline{\nu}^{-1}h^{2k}\vertiii{ u_j}^2_{2,k+1} 
\\
&\quad+C\overline{\nu}^{-1}\Delta t^2\vertiii{ \nabla u_{j,t}}^2_{2,0}+ C\overline{\nu}^{-1}h^{2k}\vertiii{ u_j}^4_{4, k+1}+C\nu^{-1}h^{2k}
 \\
&\quad+C\overline{\nu}^{-1} h^{2s+2}\Vert|p_{j}|\Vert_{2,s+1}^{2}+C\overline{\nu}^{-1} h^{2k+2}\Vert|u_{j,t}|\Vert_{2,k+1}^{2}
+C\overline{\nu}^{-1}\Delta t^{2}\Vert|u_{j,tt}|\Vert_{2,0}^{2}\Big\}
\\
&\quad+\frac{1}{2}h^{2k+2}\vertiii{ u_j }_{\infty, k+1}^2
+\left(\frac{\sqrt{\mu}}{2}\frac{(2+\epsilon)}{\sqrt{\mu}+\epsilon}-\frac{\vert \nu_j-\overline{\nu}\vert}{2 \overline{\nu}}\right)\overline{\nu}\Delta t \vertiii{ u_j}_{\infty, k+1}^2\\
&\quad+\frac{1}{15} \frac{\epsilon} {\sqrt{\mu}+\epsilon}( 1-\frac{\sqrt{\mu}}{2})\overline{\nu} h^{2k}\vertiii{ u_j}^2_{2,k+1}.
\end{aligned}
\end{equation*}
\end{theorem}
\begin{proof}
The proof is given in Appendix \ref{proofa2}.
\end{proof}

In particular, when Taylor-Hood elements ($k=2$, $s=1$) are used, i.e., the $C^{0}$ piecewise-quadratic velocity space $X_{h}$ and the $C^{0}$ piecewise-linear pressure space
$Q_{h}$, we have the following estimate.
\begin{corollary}
Assume that $\Vert e_{j}^{0}\Vert$ and $\Vert\nabla e_j^0\Vert$ are both $O(h)$ accurate or better. Then, if $(X_{h},Q_{h})$ is chosen as the $(P_2, P_1)$ Taylor-Hood element pair, we have
\begin{equation*}
\frac{1}{2}\Vert e_{j}^{N}\Vert^{2}
+\frac{1}{15} \frac{\epsilon} {\sqrt{\mu}+\epsilon}\Big( 1-\frac{\sqrt{\mu}}{2}\Big)\overline{\nu} \Delta t\sum_{n=0}^{N-1}\Vert\nabla
e_{j}^{n+1}\Vert^{2} 
\leq C (h^2 + \Delta t^2 +h \Delta t)\text{ .}
\end{equation*}
\end{corollary}

\section{Numerical experiments}\label{sec:num}

In this section, we illustrate the effectiveness of our proposed method \eqref{First-Order} and the associated theoretical analyses in \S\ref{sec:stab} and \S\ref{sec:err} by considering two examples: a Green-Taylor vortex problem and a flow between two offset cylinders. The first problem has a known exact solution that is used to illustrate the error analysis. The second example does not have an analytic solution but has complex flow structures; it is used to check the stability analysis. 

\subsection{The Green-Taylor vortex problem}

The Green-Taylor vortex flow is commonly used for testing convergence rates, e.g., see \cite{BBG07,B05,Cho68,JT15,JL02,JT15,Tafti}. 
The Green-Taylor vortex solution given by
\begin{align}
&  u(x,y,t)=-\cos(\omega\pi x)\sin(\omega\pi y)e^{-2\omega^{2}\pi^{2}t/\tau
}\nonumber\\
&  v(x,y,t)=\sin(\omega\pi x)\cos(\omega\pi y)e^{-2\omega^{2}\pi^{2}t/\tau
}\label{eq:GreenTaylorVortex}\\
&  p(x,y,t)=-\frac{1}{4}(\cos(2\omega\pi x)+\cos(2\omega\pi y))e^{-4\omega
^{2}\pi^{2}t/\tau} \nonumber
\end{align}
satisfies the NSE in $\Omega=(0,1)^{2}$ for $\tau=Re$ and initial condition 
\[
u^{0}=\big(-\cos(\omega\pi x)\sin(\omega\pi y), \sin(\omega\pi x)\cos(\omega\pi
y)\big)^{\top}.
\]
The solution consists of an
$\omega\times\omega$ array of oppositely signed vortices that decay as
$t\rightarrow\infty$. 
In the following numerical tests, we take $\omega=1$, $\nu=1/Re$, $T=1$, $h=1/m$, and $\Delta t/h=2/5$. The boundary condition is assumed to be inhomogeneous Dirichlet, that is, the boundary values match that of the exact solution.  
%Convergence rates are calculated from the error at two successive
%values of $h$ in the usual manner by postulating $e(h)=Ch^{\beta}$ and solving
%for $\beta$\ via $\beta=\ln(e(h_{1})/e(h_{2}))/\ln(h_{1}/h_{2})$. 

We consider an ensemble of two members, $u_{1}$ and $u_{2}$, corresponding to two incompressible NSE simulations with different viscosity coefficients $\nu_j$ and initial conditions $u_{j, 0}$. 
We investigate the ensemble simulations and compare it with independent simulations. 
For $j=1, 2$, we define by $\mathcal{E}^{E}_j = u_{j}-u_{j, h}$ the approximation error of the $j$-th member of the ensemble simulation and by
$\mathcal{E}^{S}_j = u_{j}-u_{j, h}$ the approximation error of the $j$-th independently determined simulation. Here, the superscript ``$E$" stands for ``ensemble" whereas ``$S$" stands for ``independent."

\paragraph{Case 1}
We set the viscosity coefficient $\nu_1=0.2$ and initial condition $u_{1, 0}= (1+\epsilon)u^0$ for the first member $u_1$ and $\nu_2=0.3$ and $u_{2, 0}= (1-\epsilon)u^0$ for the second member $u_2$, where $\epsilon= 10^{-3}$. 
For this choice of parameters, we have $\vert \nu_j -\overline{\nu}\vert/\overline{\nu} =\frac{1}{5}$ for both $j= 1$ and $j=2$ so that the condition  \eqref{ineq:CFL-h2} is satisfied. 
We first apply the ensemble algorithm; results are shown in Table \ref{tab:t6ensemble}. 
It is seen that the convergence rate for $u_{1}$ and $u_{2}$ is first order. 
%%% ensemble
\begin{table}[htp]
\centering
{\small
\caption{\noindent For the Green-Taylor vortex problem (Case 1) and for a sequence of uniform grid sizes $h$,  errors for ensemble simulations of two members with inputs $\nu_1=0.2$, $u_{1, 0}= (1+10^{-3})u^0$ and $\nu_2=0.3$, $u_{2, 0}= (1-10^{-3})u^0$.}
\label{tab:t6ensemble}%
\begin{tabular}{|c||c|c||c|c||c|c||c|c|}
\hline
$1/h$& $\|\mathcal{E}^{E}_1\|_{\infty, 0}$ & rate & $\|\nabla \mathcal{E}^{E}_1\|_{2,0}$ & rate & $\|\mathcal{E}^{E}_2\|_{\infty, 0}$ & rate & $\|\nabla \mathcal{E}^{E}_2\|_{2,0}$ & rate\\
\hline
$20 $ & $1.05\cdot 10^{-2}$ & --     & $4.17\cdot 10^{-2}$ & --& $7.36\cdot 10^{-3}$ & --     & $2.53\cdot 10^{-2}$ & -- \\
$40 $ & $5.86\cdot10^{-3}$ & 0.85 & $2.21\cdot 10^{-2}$ & 0.91 & $3.87\cdot 10^{-3}$ & 0.93  & $1.31\cdot 10^{-2}$ & 0.95  \\
$80 $ & $3.10\cdot10^{-3}$ & 0.92 & $1.14\cdot 10^{-2}$ & 0.95 & $2.02\cdot 10^{-3}$ & 0.94  & $6.70\cdot 10^{-3}$ & 0.97 \\
$160 $ & $1.59\cdot10^{-3}$ & 0.96 & $5.81\cdot 10^{-3}$ & 0.97 & $1.03\cdot 10^{-3}$ & 0.97  & $3.39\cdot 10^{-3}$ & 0.98 \\
\hline
\end{tabular}
}
\end{table}

We next compare the ensemble simulations with independent simulations. To this end, we perform each NSE simulation independently using the same discretization setup. The associated approximation errors are listed in Table \ref{tab:t6u1}. 
Comparing with Table \ref{tab:t6ensemble}, we observe that the ensemble simulation is able to achieve  accuracies close to that of the independent, more costly simulations.

\begin{table}[htp]
\centering
{\small
\caption{\noindent For the Green-Taylor vortex problem (Case 1) and for a sequence of uniform grid sizes $h$,  errors in independent simulations of two members with inputs $\nu_1=0.2$, $u_{1, 0}= (1+10^{-3})u^0$ and $\nu_2=0.3$, $u_{2, 0}= (1-10^{-3})u^0$.}
\label{tab:t6u1}%
\begin{tabular}{|c||c|c||c|c||c|c||c|c|}
\hline
$1/h$& $\|\mathcal{E}^{S}_1\|_{\infty, 0}$ & rate & $\|\nabla \mathcal{E}^{S}_1\|_{2,0}$ & rate & $\|\mathcal{E}^{S}_2\|_{\infty, 0}$ & rate & $\|\nabla \mathcal{E}^{S}_2\|_{2,0}$ & rate\\
\hline
$20 $ & $1.01\cdot 10^{-2}$ & --     & $3.88\cdot 10^{-2}$ & --  & $7.88\cdot 10^{-3}$ & --     & $2.76\cdot 10^{-2}$ & --\\
$40 $ & $5.47\cdot 10^{-3}$ & 0.89 & $2.04\cdot 10^{-2}$ & 0.93  & $4.24\cdot 10^{-3}$ & 0.90  & $1.44\cdot 10^{-2}$ & 0.93\\
$80 $ & $2.85\cdot 10^{-3}$ & 0.94 & $1.05\cdot 10^{-2}$ & 0.96 &  $2.22\cdot10^{-3}$  & 0.93 & $7.41\cdot 10^{-3}$ &  0.96\\
$160 $ & $1.46\cdot 10^{-3}$ & 0.97 & $5.30\cdot 10^{-3}$ & 0.98  & $1.13\cdot10^{-3}$  & 0.97 & $3.76\cdot 10^{-3}$ &  0.98\\
\hline
\end{tabular}
}
\end{table}

%\begin{table}[htp]
%\centering
%{\small
%\caption{Individual simulation: $\nu= 0.2$ and $u_{1, 0}= (1+10^{-3})u^0$.}
%\label{tab:t6u10}%
%\begin{tabular}{|c|c|c|c|c|}\hline
%$m$ & $\|\mathcal{E}^{I}_1\|_{\infty, 0}$ & rate & $\|\nabla \mathcal{E}^{I}_1\|_{2,0}$ & rate\\
%\hline
%$20 $ & $1.01\cdot 10^{-2}$ & --     & $3.88\cdot 10^{-2}$ & --\\
%$40 $ & $5.47\cdot 10^{-3}$ & 0.89 & $2.04\cdot 10^{-2}$ & 0.93 \\
%$80 $ & $2.85\cdot 10^{-3}$ & 0.94 & $1.05\cdot 10^{-2}$ & 0.96 \\
%$160 $ & $1.46\cdot 10^{-3}$ & 0.97 & $5.30\cdot 10^{-3}$ & 0.98 \\
%\hline
%\end{tabular}
%}
%\end{table}

%
%\begin{table}[htp]
%\centering
%{\small
%\caption{Individual simulation: $\nu= 0.3$ and $u_{2, 0}= (1-10^{-3})u^0$.}
%\label{tab:t6u2}%
%\begin{tabular}{|c|c|c|c|c|}
%\hline
%$m$ & $\|\mathcal{E}^{I}_2\|_{\infty, 0}$ & rate & $\|\nabla \mathcal{E}^{I}_2\|_{2,0}$ & rate\\
%\hline
%$20 $ & $7.88\cdot 10^{-3}$ & --     & $2.76\cdot 10^{-2}$ & --\\
%$40 $ & $4.24\cdot 10^{-3}$ & 0.90  & $1.44\cdot 10^{-2}$ & 0.93  \\
%$80 $ & $2.22\cdot10^{-3}$  & 0.93 & $7.41\cdot 10^{-3}$ &  0.96 \\
%$160 $ & $1.13\cdot10^{-3}$  & 0.97 & $3.76\cdot 10^{-3}$ &  0.98 \\
%\hline
%\end{tabular}
%}
%\end{table}

\paragraph{Case 2}
We now set $\nu_1= 0.01$ and $\nu_2= 0.49$ while keeping the same initial conditions as for Case 1. With this choice of parameters, $\vert \nu_j -\overline{\nu}\vert/\overline{\nu} =\frac{24}{25}$ for both $j= 1$ and $j= 2$, which still satisfies \eqref{ineq:CFL-h2} but is closer to the upper limit. 
The ensemble simulation errors are listed in Table \ref{tab:t2ensemble}, which shows the rate of convergence for the second member is nearly 1 and for the first member is approaching 1. 

%%% ensemble
\begin{table}[htp]
\centering
{\small
\caption{\noindent For the Green-Taylor vortex problem (Case 2) and for a sequence of uniform grid sizes $h$, errors in ensemble simulations of two members: $\nu_1=0.01$, $u_{1, 0}= (1+10^{-3})u^0$ and $\nu_2=0.49$, $u_{2, 0}= (1-10^{-3})u^0$.}
\label{tab:t2ensemble}%
\begin{tabular}{|c|c|c|c|c|c|c|c|c|}
\hline
$1/h$& $\|\mathcal{E}^{E}_1\|_{\infty, 0}$ & rate & $\|\nabla \mathcal{E}^{E}_1\|_{2,0}$ & rate & $\|\mathcal{E}^{E}_2\|_{\infty, 0}$ & rate & $\|\nabla \mathcal{E}^{E}_2\|_{2,0}$ & rate\\
\hline
$20   $ & $2.91\cdot 10^{-2}$ & --     & $2.96\cdot 10^{-1}$ & --& $3.50\cdot 10^{-3}$ & --     & $9.94\cdot 10^{-3}$ & -- \\
$40 $ & $1.86\cdot10^{-2}$ & 0.65  & $1.80\cdot 10^{-1}$ & 0.71 & $1.65\cdot 10^{-3}$ & 1.08 & $4.97\cdot 10^{-3}$ & 1 \\
$80 $ & $1.08\cdot10^{-2}$ & 0.78 & $1.02\cdot 10^{-1}$ & 0.83 & $8.53\cdot 10^{-4}$ & 0.95 & $2.52\cdot 10^{-3}$ & 0.98  \\
$160 $ & $5.89\cdot10^{-3}$ & 0.87  & $5.46\cdot 10^{-2}$ & 0.90 & $4.32\cdot 10^{-4}$ & 0.98 & $1.27\cdot 10^{-3}$ & 0.98  \\
\hline
\end{tabular}
}
\end{table}

The approximation errors for two independent simulations under using the same discretization setup are listed in Table \ref{tab:t2u1}. 
Comparing the ensemble simulation results in Table \ref{tab:t2ensemble} with the independent simulations, we find that the accuracy of first member in the ensemble simulation degrades slightly whereas that of the second member in the ensemble simulation improves a bit. 
Overall, the ensemble simulation is able to achieve the same order of accuracy as the independent simulations.

\begin{table}[htp]
\centering
{\small
\caption{\noindent For the Green-Taylor vortex problem (Case 2) and for a sequence of uniform grid sizes $h$, errors in independent simulations of two members: $\nu_1=0.01$, $u_{1, 0}= (1+10^{-3})u^0$ and $\nu_2=0.49$, $u_{2, 0}= (1-10^{-3})u^0$.}
\label{tab:t2u1}%
\begin{tabular}{|c|c|c|c|c|c|c|c|c|}
\hline
$1/h$& $\|\mathcal{E}^{S}_1\|_{\infty, 0}$ & rate & $\|\nabla \mathcal{E}^{S}_1\|_{2,0}$ & rate & $\|\mathcal{E}^{S}_2\|_{\infty, 0}$ & rate & $\|\nabla \mathcal{E}^{S}_2\|_{2,0}$ & rate\\
\hline
$20   $ & $3.19\cdot 10^{-2}$ & --     & $2.95\cdot 10^{-1}$ & --& $5.49\cdot 10^{-3}$ & --     & $1.79\cdot 10^{-2}$ & --\\
$40 $ & $1.67\cdot10^{-2}$ & 0.93 & $1.54\cdot 10^{-1}$ & 0.93 & $3.03\cdot 10^{-3}$ & 0.86  & $9.38\cdot 10^{-3}$ & 0.94 \\
$80 $ & $8.56\cdot10^{-3}$ & 0.97 & $7.90\cdot 10^{-2}$ & 0.97 & $1.59\cdot10^{-3}$ &  0.93 & $4.81\cdot 10^{-3}$ &  0.96\\
$160 $ & $4.33\cdot10^{-3}$ & 0.98 & $3.99\cdot 10^{-2}$ & 0.98 & $8.18\cdot10^{-4}$ &  0.96 & $2.44\cdot 10^{-3}$ &  0.98\\
\hline
\end{tabular}
}
\end{table}

%\begin{table}[htp]
%\centering
%{\small
%\caption{Individual simulation: $\nu= 0.01$ and $u_{1, 0}= (1+10^{-3})u^0$.}
%\label{tab:t2u1}%
%\begin{tabular}{|c|c|c|c|c|}\hline
%$m$ & $\|\mathcal{E}^{I}_1\|_{\infty, 0}$ & rate & $\|\nabla \mathcal{E}^{I}_1\|_{2,0}$ & rate\\
%\hline
%$20   $ & $3.19\cdot 10^{-2}$ & --     & $2.95\cdot 10^{-1}$ & --\\
%$40 $ & $1.67\cdot10^{-2}$ & 0.93 & $1.54\cdot 10^{-1}$ & 0.93  \\
%$80 $ & $8.56\cdot10^{-3}$ & 0.97 & $7.90\cdot 10^{-2}$ & 0.97 \\
%$160 $ & $4.33\cdot10^{-3}$ & 0.98 & $3.99\cdot 10^{-2}$ & 0.98 \\
%\hline
%\end{tabular}
%}
%\end{table}

%
%\begin{table}[htp]
%\centering
%{\small
%\caption{Individual simulation: $\nu= 0.49$ and $u_{2, 0}= (1-10^{-3})u^0$.}
%\label{tab:t2u2}%
%\begin{tabular}{|c|c|c|c|c|}
%\hline
%$m$ & $\|\mathcal{E}^{I}_2\|_{\infty, 0}$ & rate & $\|\nabla \mathcal{E}^{I}_2\|_{2,0}$ & rate\\
%\hline
%$20  $ & $5.49\cdot 10^{-3}$ & --     & $1.79\cdot 10^{-2}$ & --\\
%$40 $ & $3.03\cdot 10^{-3}$ & 0.86  & $9.38\cdot 10^{-3}$ & 0.94  \\
%$80 $ & $1.59\cdot10^{-3}$ &  0.93 & $4.81\cdot 10^{-3}$ &  0.96 \\
%$160 $ & $8.18\cdot10^{-4}$ &  0.96 & $2.44\cdot 10^{-3}$ &  0.98 \\
%\hline
%\end{tabular}
%}
%\end{table}

%\newpage
\subsection{Flow between two offset cylinders}
Next, we check the stability of our algorithm by considering the problem of a flow between two offset circles \cite{J15, JL14, JL15, JKL15}. The domain is a disk with a smaller off center
obstacle inside. Letting $r_{1}=1$, $r_{2}=0.1$, and $c=(c_{1},c_{2})=(\frac{1}{2}
,0)$, the domain is given by
\[
\Omega=\{(x,y)\,\,:\,\,x^{2}+y^{2}\leq r_{1}^{2} \,\,\text{ and }\,\, (x-c_{1})^{2}%
+(y-c_{2})^{2}\geq r_{2}^{2}\}.
\]
\begin{figure}[h]
\begin{center}
\includegraphics[width=.6\textwidth]{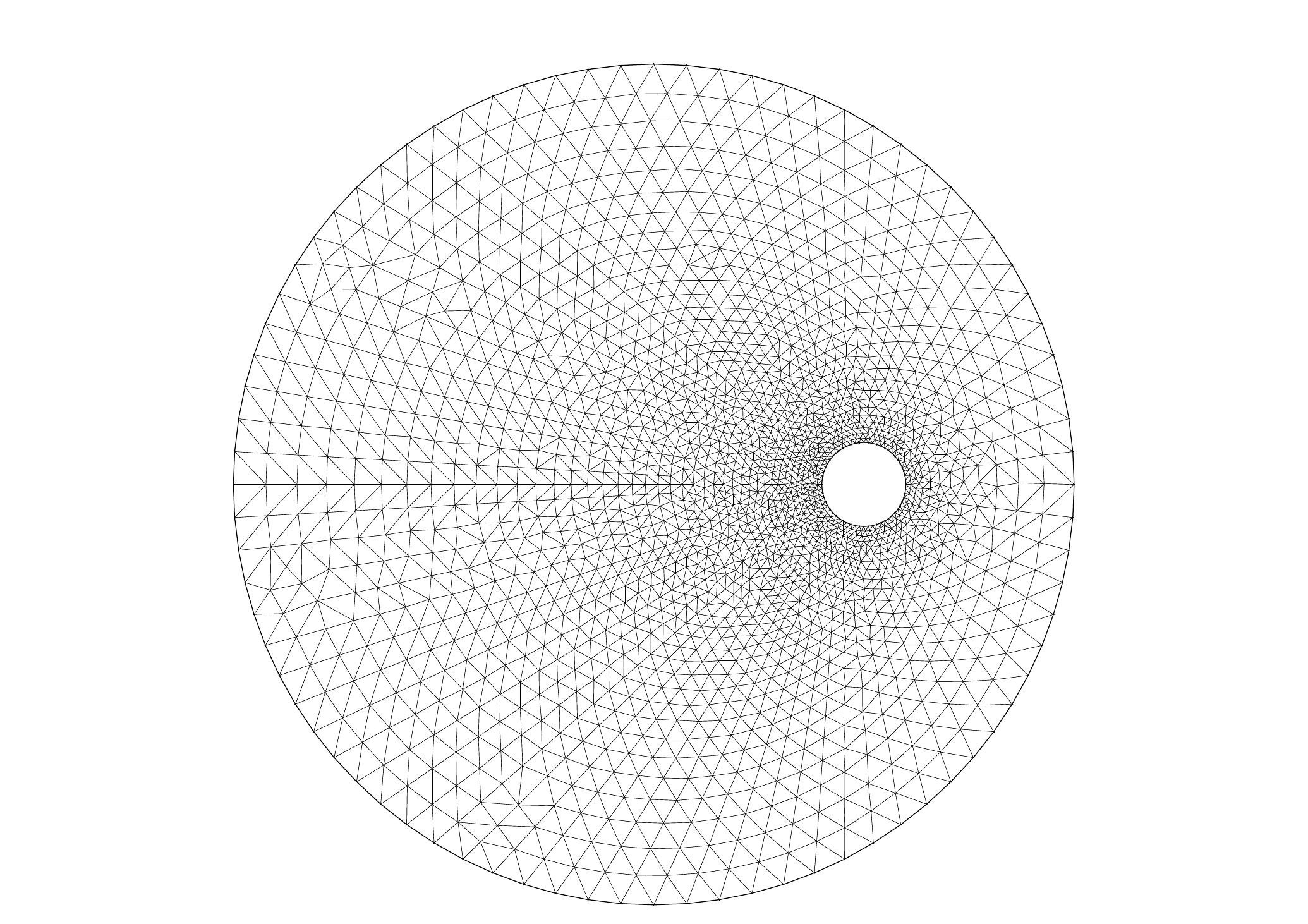}
\end{center}
\caption{Mesh for the flow between two offset cylinders example.}
\label{mesh}
\end{figure}
\noindent The flow is driven by a counterclockwise rotational body force
\[
f(x,y,t)=\big(-6y(1-x^{2}-y^{2}), 6x(1-x^{2}-y^{2})\big)^\top
\]
with no-slip boundary conditions imposed on both circles. The flow between the two circles shows interesting structures interacting with the inner circle. A Von
K$\acute{a}$rm$\acute{a}$n vortex street is formed behind the inner circle and then re-interacts with that circle and with itself, generating complex flow patterns.
We consider multiple numerical simulations of the flow with different viscosity coefficients using the ensemble-based algorithm \eqref{First-Order-h}. 
For spatial discretization, we apply the Taylor-Hood element pair on a triangular mesh that is generated by Delaunay triangulation with $80$ mesh points on the outer circle and $60$ mesh points on the inner circle and with refinement near the inner circle, resulting in $18,638$ degrees of freedom; see Figure \ref{mesh}.

In order to illustrate the stability analysis, we select two different sets of viscosity coefficients for:
\begin{itemize}
\item[] Case 1: \quad $\nu_1=0.005$,\,\, $\nu_2=0.039$,\,\, $\nu_3=0.016$
\item[] Case 2: \quad $\nu_1=0.005$,\,\, $\nu_2=0.041$,\,\, $\nu_3=0.014$.
\end{itemize}
The average of the viscosity coefficients is $\overline{\nu}= 0.02$ for both cases. 
However, the stability condition \eqref{ineq:CFL-h2} is satisfied in the first case but is not satisfied in the second one, i.e., we have 
\begin{itemize}
\item[] Case 1: \quad $\frac{\vert\nu_1-\overline{\nu}\vert}{\overline{\nu}}=\frac{3}{4}$,\,\, $\frac{\vert\nu_2-\overline{\nu}\vert}{\overline{\nu}}=\frac{19}{20}$,\,\,
$\frac{\vert\nu_3-\overline{\nu}\vert}{\overline{\nu}}=\frac{1}{5}$
\item[] Case 2: \quad   $\frac{\vert\nu_1-\overline{\nu}\vert}{\overline{\nu}}=\frac{3}{4}$,\,\, $\frac{\vert\nu_2-\overline{\nu}\vert}{\overline{\nu}}=\frac{21}{20}$,\,\,
$\frac{\vert\nu_3-\overline{\nu}\vert}{\overline{\nu}}=\frac{3}{10}$.
\end{itemize}
The second member of Case 2 has a perturbation ratio greater than 1. 
Simulations of both cases are subject to the same initial condition and body forces for all ensemble members. 
In particular, the initial condition is generated by solving the steady Stokes problem with viscosity $\nu=0.02$ and the same body force $f(x, y, t)$. 
All the simulations are run over the time interval $[0, 5]$ with a time step size $\Delta t= 0.01$. 
For the stability test, we use the kinetic energy as a criterion and compare the ensemble simulation results with independent simulations using the same mesh and time-step size.

The comparison of the energy evolution of ensemble-based simulations with the corresponding independent simulations is shown in Figures \ref{egy_1} and \ref{egy_2}. 
It is seen that, for Case 1, the ensemble simulation is stable, but for Case 2, it becomes unstable. 
This phenomena coincides with our stability analysis since the condition \eqref{ineq:CFL-h2} holds for all members of Case 1, but does not hold for the second member of Case 2. 
Indeed, it is observed from Figure \ref{egy_2} that the energy of the second member in Case 2 blows up after $t=3.7$, then affecting other two members and results in their  energy dramatically increase after $t=4.7$. 

\begin{figure}[h!]
\begin{center}
\includegraphics[width=.8\textwidth]{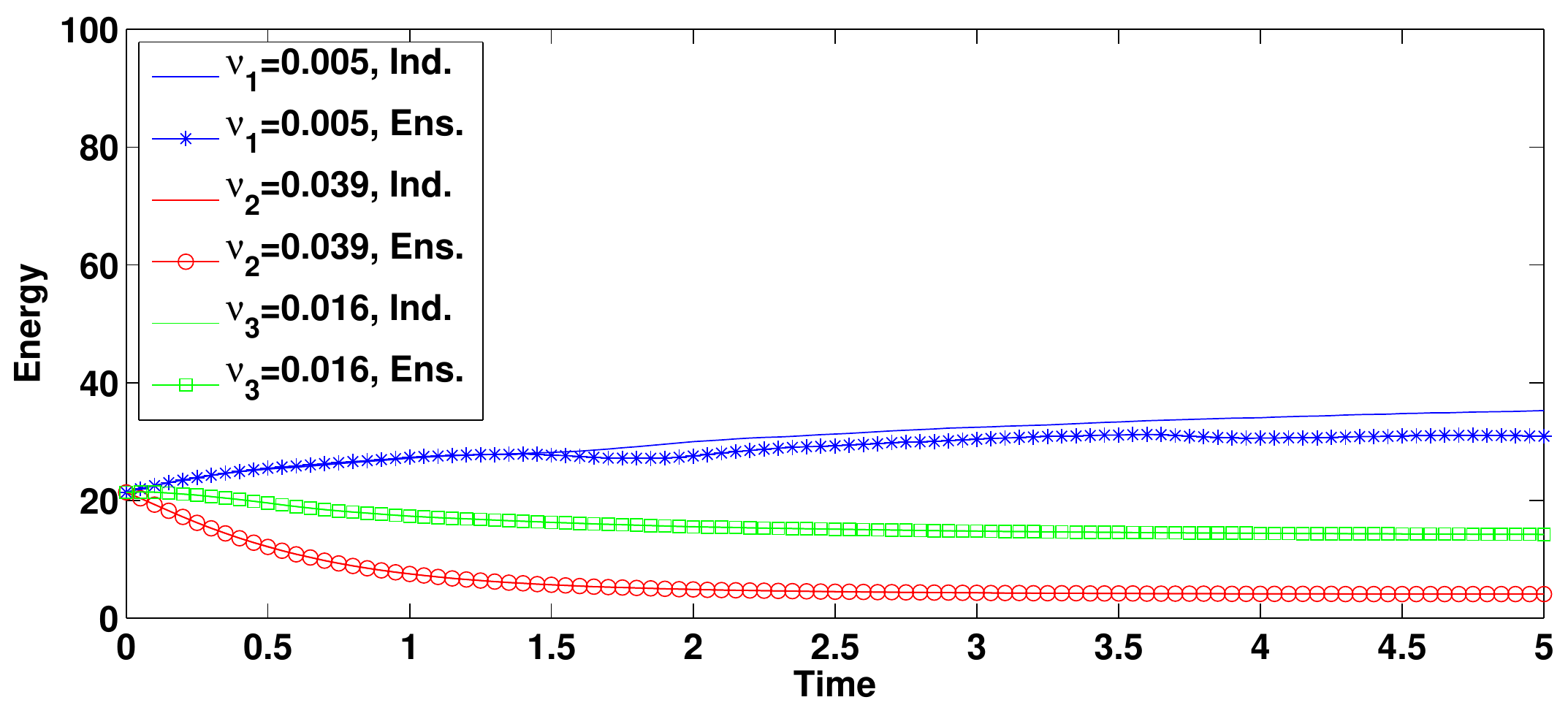}
\end{center}
\caption{For the flow between two offset cylinders, Case 1, the energy evolution of the ensemble (Ens.) and independent simulations (Ind.).}
\label{egy_1}
\end{figure}
\begin{figure}[h!]
\begin{center}
\includegraphics[width=.8\textwidth]{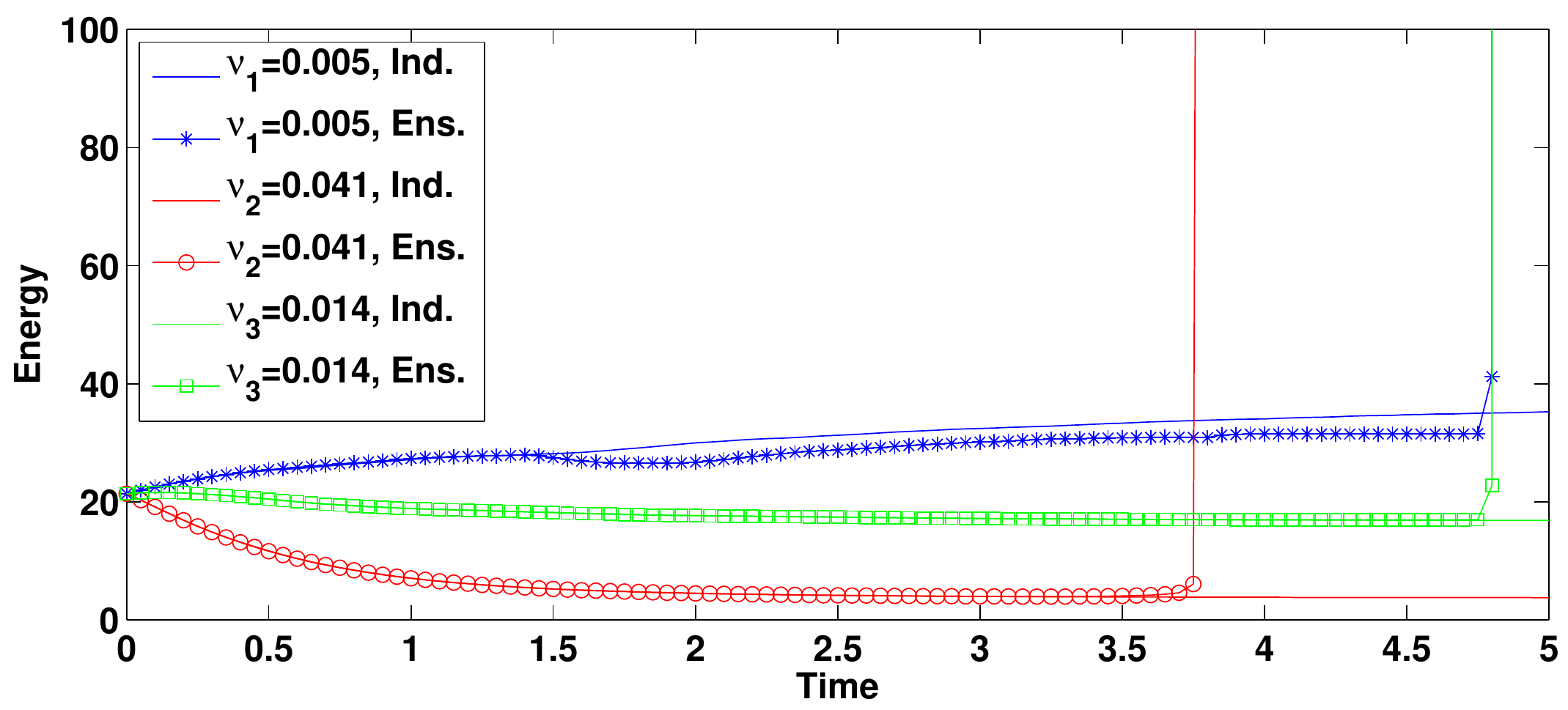}
\end{center}
\caption{For the flow between two offset cylinders, Case 2, the energy evolution of the ensemble (Ens.) and independent simulations (Ind.).}
\label{egy_2}
\end{figure}

%Note that, as shown in Figure \ref{egy_1}, an energy discrepancy between the ensemble simulation and corresponding individual simulation appears after $t=1.5$ when $\nu_1= 0.005$. This is due to the error accumulation and could be improved by using a higher order scheme, which is one of our future work. 
\section{Conclusions}\label{sec:con}
In this paper, we consider a set of Navier-Stokes simulations in which each member may be subject to a distinct viscosity coefficient, initial conditions, and/or body forces. 
An ensemble algorithm is developed for the group by which all the flow ensemble members, after discretization, share a common linear system with different right-hand side vectors. 
This leads to great saving in both storage requirements and computational costs. The stability and accuracy of the ensemble method are analyzed.
Two numerical experiments are presented. 
The first is for Green-Taylor flow and serves to illustrate the first-order accuracy in time of the ensemble-based scheme. The second is for a flow between two offset cylinders and serves to show that our stability analysis is sharp. 
As a next step, we will investigate higher-order accurate schemes for the flow ensemble simulations.

\appendix

\section{Proof of Theorem \ref{th:First-Order}}\label{proofa1}
{\allowdisplaybreaks
%Under the assumption of the LBB$_h$ condition, we restrict $v_h$ in \eqref{First-Order-h} to the space of discretely divergence free functions $V^h$.     
Setting $v_{h}=u_{j,h}^{n+1}$ and $q_h = p_{j, h}^{n+1}$ in (\ref{First-Order-h}) and then adding two equations, we obtain
\begin{equation*}
\begin{aligned}
\frac{1}{2}\Vert u_{j,h}^{n+1}\Vert^{2}&-\frac{1}{2}\Vert u_{j,h}^{n}\Vert
^{2}+\frac{1}{2}\Vert u_{j,h}^{n+1}-u_{j,h}^{n}\Vert^{2}+ \Delta t b^{*}(u_{j,h}^{n}-\overline{u}_{h}^{n}, u_{j,h}^{n},u_{j,h}^{n+1}) \\
&+
\overline{\nu}\Delta t \Vert\nabla u_{j,h}^{n+1}\Vert^{2}= \Delta t (f_{j}^{n+1},
u_{j,h}^{n+1})-\left( \nu_j - \overline{\nu}\right) \Delta t \left( \nabla u_{j,h}^n, \nabla u_{j,h}^{n+1}\right).
\end{aligned}
\end{equation*}
Applying Young's inequality to the terms on the RHS yields, for $\forall\, \alpha, \beta >0$,
\begin{align}
&\frac{1}{2}\Vert u_{j,h}^{n+1}\Vert^{2}-\frac{1}{2}\Vert u_{j,h}^{n}\Vert
^{2}+\frac{1}{2}\Vert u_{j,h}^{n+1}-u_{j,h}^{n}\Vert^{2}+ \overline{\nu}\Delta t \Vert\nabla u_{j,h}^{n+1}\Vert^{2} \nonumber\\
&\hspace{5cm}+ \Delta t b^{*}(u_{j,h}^{n}-\overline{u}_{h}^{n},u_{j,h}^{n},u_{j,h}^{n+1}-u_{j,h}%
^{n})\label{ineq: tri}\\
&\leq\frac{\alpha\overline{\nu}\Delta t}{4} \Vert\nabla u_{j,h}^{n+1} \Vert^{2}+ \frac{\Delta
t}{ \alpha\overline{\nu}} \Vert f_{j}^{n+1}\Vert_{-1}^{2}+\frac{\beta\overline{\nu}\Delta t}{4} \Vert\nabla u_{j,h}^{n+1} \Vert^{2}+\frac{(\nu_j-\overline{\nu})^2\Delta t}{\beta\overline{\nu}}\Vert \nabla u_{j,h}^n\Vert^2. \nonumber
\end{align}
Both $\frac{\beta\overline{\nu}\Delta t}{4} \Vert\nabla u_{j,h}^{n+1} \Vert^{2}$ and $\frac{(\nu_j-\overline{\nu})^2\Delta t}{\beta\overline{\nu}}\Vert \nabla u_{j,h}^n\Vert^2$ on the RHS of \eqref{ineq: tri} need to be absorbed into $\overline{\nu}\Delta t \Vert\nabla u_{j,h}^{n+1}\Vert^{2}$ on the LHS. To this end, we minimize
%the amount of $\Vert \nabla u_{j,h}^{n+1}\Vert^2$ needed, i.e. 
$\frac{\beta\overline{\nu}\Delta t}{4}+\frac{(\nu_j-\overline{\nu})^2\Delta t}{\beta\overline{\nu}}$ by selecting $\beta = \frac{2\vert \nu_j-\overline{\nu}\vert }{\overline{\nu} }$ so that \eqref{ineq: tri} becomes
%%%%%%%%%%%%%
%\newpage
%%%%%%%%%%%%%
\begin{equation}
\label{ineq:s0}
\begin{aligned}
&\frac{1}{2}\Vert u_{j,h}^{n+1}\Vert^{2}-\frac{1}{2}\Vert u_{j,h}^{n}\Vert
^{2}+\frac{1}{2}\Vert u_{j,h}^{n+1}-u_{j,h}^{n}\Vert^{2}+ \overline{\nu}\Delta t \Vert\nabla u_{j,h}^{n+1}\Vert^{2} \\
&\qquad+ \Delta t b^{*}(u_{j,h}^{n}-\overline{u}_{h}^{n},u_{j,h}^{n},u_{j,h}^{n+1}-u_{j,h}%
^{n})\leq\frac{\alpha\overline{\nu}\Delta t}{4} \Vert\nabla u_{j,h}^{n+1} \Vert^{2}\\
&\qquad+ \frac{\Delta
t}{ \alpha\overline{\nu}} \Vert f_{j}^{n+1}\Vert_{-1}^{2}
+\frac{\vert \nu_j-\overline{\nu}\vert\Delta t}{2} \Vert\nabla u_{j,h}^{n+1} \Vert^{2}+\frac{\vert \nu_j-\overline{\nu}\vert \Delta t}{2}\Vert \nabla u_{j,h}^n\Vert^2.
\end{aligned}
\end{equation}
Next, we bound the trilinear term using the inequality \eqref{In2} and
the inverse inequality \eqref{inverse}, obtaining
\begin{equation*}
\begin{aligned}
&- \Delta t b^{*}(u_{j,h}^{n}-\overline{u}_{h}^{n},u_{j,h}^{n},u_{j,h}^{n+1}-u_{j,h}%
^{n})\\
&\qquad\quad\leq C\Delta t\Vert\nabla(u_{j,h}^{n}-\overline{u}_{h}^{n})\Vert\Vert\nabla u_{j,h}%
^{n}\Vert\left(\Vert \nabla ( u_{j,h}^{n+1}-u_{j,h}^{n})\Vert\Vert u_{j,h}^{n+1}-u_{j,h}^{n}\Vert\right)^{1/2}\\
&\qquad\quad\leq C\Delta t\Vert\nabla(u_{j,h}^{n}-\overline{u}_{h}^{n})\Vert\Vert\nabla u_{j,h}%
^{n}\Vert(Ch^{-\frac{1}{2}})\Vert u_{j,h}^{n+1}-u_{j,h}^{n}\Vert.
\end{aligned}
\end{equation*}
Using Young's inequality again gives 
\begin{equation}
\label{ineq:s1}
\begin{aligned}
&- \Delta t b^{*}(u_{j,h}^{n}-\overline{u}_{h}^{n},u_{j,h}^{n},u_{j,h}^{n+1}-u_{j,h}%
^{n})\\
&\qquad\quad\leq C \frac{\Delta t^{2}}{h} \Vert\nabla(u_{j,h}^{n}-\overline{u}_{h}^{n})\Vert^{2}
\Vert\nabla u_{j,h}^{n}\Vert^{2}+ \frac{1}{4}\Vert u_{j,h}^{n+1}-u_{j,h}%
^{n}\Vert^{2}\text{ .}
\end{aligned}
\end{equation}
Substituting \eqref{ineq:s1} into \eqref{ineq:s0} and combining like terms, we have 
\begin{equation*}
\begin{aligned}
&\hspace{-10pt}\frac{1}{2}\Vert u_{j,h}^{n+1}\Vert^{2}-\frac{1}{2}\Vert u_{j,h}^{n}%
||^{2}+\frac{1}{4}\Vert u_{j,h}^{n+1}-u_{j,h}^{n}\Vert^{2}+ \overline{\nu}\Delta
t \left(1-\frac{\alpha}{4}-\frac{\vert \nu_j-\overline{\nu}\vert}{2 \overline{\nu}}\right)\Vert\nabla u_{j,h}^{n+1}\Vert^{2}\\
&\quad\leq\frac{\Delta t}{\alpha\overline{\nu}} \Vert f_{j}^{n+1}\Vert_{-1}^{2} + C \frac{\Delta
t^{2}}{h} \Vert\nabla( u_{j,h}^{n}-\overline{u}_{h}^{n})\Vert^{2} \Vert\nabla
u_{j,h}^{n}\Vert^{2}+\frac{\vert \nu_j-\overline{\nu}\vert \Delta t}{2}\Vert \nabla u_{j,h}^n\Vert^2.
\end{aligned}
\end{equation*}
For any $0<\sigma <1$, 
\begin{equation}
\label{ineq:stable01}
\begin{aligned}
\frac{1}{2}&\Vert u_{j,h}^{n+1}\Vert^{2}-\frac{1}{2}\Vert u_{j,h}^{n}\Vert
^{2}+\frac{1}{4}\Vert u_{j,h}^{n+1}-u_{j,h}^{n}\Vert^{2}\\
&+ \overline{\nu}\Delta t \Big(1-\frac{\alpha}{4}-\frac{\vert \nu_j-\overline{\nu}\vert}{2 \overline{\nu}}\Big)\left(\Vert\nabla u_{j,h}^{n+1}\Vert^{2}-\Vert\nabla u_{j,h}^{n}\Vert
^{2}\right)\\
&+\overline{\nu}\Delta t \left((1-\sigma)\Big(1-\frac{\alpha}{4}-\frac{\vert \nu_j-\overline{\nu}\vert}{2 \overline{\nu}}\Big)-\frac{C \Delta
t}{\overline{\nu} h}\Vert\nabla(u_{j,h}^{n}-\overline{u}_{h}^{n})\Vert^{2}\right) \Vert\nabla
u_{j,h}^{n}\Vert^{2}\\
&+\overline{\nu} \Delta t \left(\sigma\Big(1-\frac{\alpha}{4}-\frac{\vert \nu_j-\overline{\nu}\vert}{2 \overline{\nu}}\Big)-\frac{\vert \nu_j-\overline{\nu}\vert}{2 \overline{\nu}}\right)\Vert \nabla u_{j,h}^n\Vert^2\leq\frac{\Delta t}{ \alpha\overline{\nu}} \Vert f_{j}^{n+1}\Vert
_{-1}^{2}\text{ .}
\end{aligned}
\end{equation}
Select $\alpha=4-\frac{2(\sigma+1)}{\sigma}\sqrt{\mu}$. Since $\alpha$ is supposed to be greater than $0$, we have
\begin{align}
\sigma >\frac{\sqrt{\mu}}{2-\sqrt{\mu}} \in (0,1)
\end{align}
%Now taking $\sigma= \frac{\sqrt{\mu}}{2-\sqrt{\mu}} + \frac{1}{2}(1-\frac{\sqrt{\mu}}{2-\sqrt{\mu}})= \frac{1}{2}+\frac{\sqrt{\mu}}{4-2\sqrt{\mu}} = \frac{1}{2-\sqrt{\mu}}\in (\frac{1}{2},1)$, 
Now taking $\sigma= \frac{\sqrt{\mu}+\epsilon}{2-\sqrt{\mu}}$, where $\epsilon \in (0, 2-2\sqrt{\mu})$ , 
\eqref{ineq:stable01} becomes
\begin{equation}
\label{ineq:stable001}
\begin{aligned}
\frac{1}{2}&\Vert u_{j,h}^{n+1}\Vert^{2}-\frac{1}{2}\Vert u_{j,h}^{n}\Vert
^{2}+\frac{1}{4}\Vert u_{j,h}^{n+1}-u_{j,h}^{n}\Vert^{2}\\
&+ \overline{\nu}\Delta t \Big(1-\frac{\alpha}{4}-\frac{\vert \nu_j-\overline{\nu}\vert}{2 \overline{\nu}}\Big)\left(\Vert\nabla u_{j,h}^{n+1}\Vert^{2}-\Vert\nabla u_{j,h}^{n}\Vert
^{2}\right)\\
&+\overline{\nu}\Delta t \bigg(\Big( (1-\sigma)\frac{\sigma+1}{2\sigma}\sqrt{\mu}-(1-\sigma)\frac{\vert \nu_j-\overline{\nu}\vert}{2 \overline{\nu}}\Big)
\\&-\frac{C \Delta
t}{\overline{\nu} h}\Vert\nabla(u_{j,h}^{n}-\overline{u}_{h}^{n})\Vert^{2}\bigg) \Vert\nabla
u_{j,h}^{n}\Vert^{2}\\
&+\overline{\nu} \Delta t \Big( \frac{\sigma+1}{2}\sqrt{\mu}-(1+\sigma)\frac{\vert \nu_j-\overline{\nu}\vert}{2\overline{\nu}}\Big)\Vert \nabla u_{j,h}^n\Vert^2\leq\frac{\Delta t}{ \alpha\overline{\nu}} \Vert f_{j}^{n+1}\Vert
_{-1}^{2}\text{ .}
\end{aligned}
\end{equation}
Stability follows if the following conditions hold:
\begin{gather}
(1-\sigma)\frac{\sigma+1}{2\sigma}\sqrt{\mu}-(1-\sigma)\frac{1}{2}\frac{\vert \nu_j-\overline{\nu}\vert}{\overline{\nu}}-\frac{C \Delta t}{\overline{\nu} h}\Vert\nabla(u_{j,h}^{n}
-\overline{u}_{h}^{n})\Vert^{2} \geq0\label{cond1}\\
  \frac{\sigma+1}{2}\sqrt{\mu}-(1+\sigma)\frac{\vert \nu_j-\overline{\nu}\vert}{2\overline{\nu}}\geq 0.\label{cond2}
\end{gather}
Using assumption  (\ref{ineq:CFL-h2}), we have
\begin{gather*}
\frac{\sigma+1}{2}\sqrt{\mu}-(1+\sigma)\frac{\vert \nu_j-\overline{\nu}\vert}{2\overline{\nu}}
 = \frac{2+\epsilon}{2-\sqrt{\mu}}\left(\frac{\sqrt{\mu}}{2}-\frac{|\nu_j-\overline{\nu}|}{2\overline{\nu}}\right)
 \geq 0
%\frac{\sigma+1}{2}\sqrt{\mu}-(1+\sigma)\frac{\sqrt{\mu}}{2} \geq 0.
\end{gather*}
so that \eqref{cond1} holds.
Together with assumption  \eqref{ineq:CFL-h1}, we then have
$$
\begin{aligned}
&(1-\sigma)\frac{\sigma+1}{2\sigma}\sqrt{\mu}-(1-\sigma)\frac{1}{2}\frac{\vert \nu_j-\overline{\nu}\vert}{\overline{\nu}}-\frac{C \Delta t}{\overline{\nu} h}\Vert\nabla(u_{j,h}^{n}
-\overline{u}_{h}^{n})\Vert^{2} \\&\qquad
\geq (1-\sigma)\frac{\sigma+1}{2\sigma}\sqrt{\mu}-(1-\sigma)\frac{1}{2}\sqrt{\mu}-\frac{C \Delta t}{\overline{\nu} h}\Vert\nabla(u_{j,h}^{n}
-\overline{u}_{h}^{n})\Vert^{2} \\&\qquad
=\frac{(2-2\sqrt{\mu}-\epsilon)\sqrt{\mu}}{2(\sqrt{\mu}+\epsilon)}-\frac{(2-2\sqrt{\mu}-\epsilon)\sqrt{\mu}}{2(\sqrt{\mu}+\epsilon)}=0
\end{aligned}
$$
so that \eqref{cond2} holds. Therefore, assuming that  \eqref{ineq:CFL-h1} and \eqref{ineq:CFL-h2} hold,  (\ref{ineq:stable001}) reduces to
\begin{equation}
\label{ineq:stable02}
\begin{aligned}
\frac{1}{2}&\Vert u_{j,h}^{n+1}\Vert^{2}-\frac{1}{2}\Vert u_{j,h}^{n}\Vert
^{2}+\frac{1}{4}\Vert u_{j,h}^{n+1}-u_{j,h}^{n}\Vert^{2}%
\\
&+  
\overline{\nu}\Delta t \left(\frac{\sqrt{\mu}}{2}\frac{2+\epsilon}{\sqrt{\mu}+\epsilon}-\frac{\vert \nu_j-\overline{\nu}\vert}{2 \overline{\nu}}\right)
\left(\Vert\nabla u_{j,h}^{n+1}\Vert^{2}-\Vert\nabla u_{j,h}^{n}\Vert^{2}\right)  
\leq\frac{\Delta t}{ \overline{\nu}} \Vert f_{j}^{n+1}\Vert_{-1}^{2}\text{
.}
\end{aligned}
\end{equation}
Summing up (\ref{ineq:stable02}) from $n=0$ to $n=N-1$ results in
\begin{equation}
\begin{split}
\frac{1}{2}&\|u_{j,h}^{N}\|^{2}+\frac{1}{4}\sum_{n=0}^{N-1}\|u_{j,h}%
^{n+1}-u_{j,h}^{n}\|^{2}
+\overline{\nu}\Delta t \left(\frac{\sqrt{\mu}}{2}\frac{2+\epsilon}{\sqrt{\mu}+\epsilon}-\frac{\vert \nu_j-\overline{\nu}\vert}{2 \overline{\nu}}\right)
\|\nabla u_{j,h}^{N}\|^{2}\\
&\leq\sum_{n=0}^{N-1}\frac{\Delta t}{\overline{\nu}}\|f_{j}^{n+1}\|_{-1}^{2}+ \frac{1}%
{2}\|u_{j,h}^{0}\|^{2}
+\overline{\nu}\Delta t \left(\frac{\sqrt{\mu}}{2}\frac{2+\epsilon}{\sqrt{\mu}+\epsilon}-\frac{\vert \nu_j-\overline{\nu}\vert}{2 \overline{\nu}}\right)\|\nabla u_{j,h}^{0}\|^{2}\text{
.}
\end{split}
\label{eq:sta}
\end{equation}
This completes the proof of stability.
}

\section{Proof of Theorem \ref{th:errBEFE-Ensemble}}\label{proofa2}
{\allowdisplaybreaks
The weak solution of the NSE $u_{j}$ satisfies%
\begin{equation}
\label{eq:convtrue}
\begin{aligned}
\Big(\frac{u_{j}^{n+1}-u_{j}^{n}}{\Delta t}, v_{h}\Big)&+ b^{*}(u_{j}^{n+1},
u_{j}^{n+1}, v_{h})+ \nu_j(\nabla u_{j}^{n+1}, \nabla v_{h})- (p_{j}%
^{n+1},\nabla\cdot v_{h})\\
&=(f_{j}^{n+1}, v_{h})
+ \text{Intp}(u_{j}^{n+1};v_{h})\text{ , }\quad\forall v_{h}\in V_{h},
\end{aligned}
\end{equation}
where 
$\text{Intp}(u_{j}^{n+1};v_{h})= \big(\frac{u_{j}^{n+1}-u_{j}^{n}}{\Delta t}-u_{j,t}%
(t^{n+1}),v_{h}\big)$. 

Let
$$
e_{j}^{n}=u_{j}^{n}-u_{j,h}^{n}=(u_{j}^{n}-I_{h} u_{j}^{n})+(I_{h} u_{j}%
^{n}-u_{j,h}^{n})=\eta_{j}^{n}+\xi_{j,h}^{n} ,
$$
where $I_{h} u_{j}^{n} \in V_{h} $ is the interpolant of $u_{j}^{n}$
in $V_{h}.$
Subtracting (\ref{eq: conv}) from (\ref{eq:convtrue}) gives%
\begin{align*}
&\Big(\frac{\xi_{j,h}^{n+1}-\xi_{j,h}^{n}}{ \Delta t},v_{h}\Big) +\overline{\nu}(\nabla\xi
_{j,h}^{n+1},\nabla v_{h})+(\nu_j-\overline{\nu})(\nabla (u_{j}^{n+1}-u_{j}^n),\nabla v_{h})\nonumber\\
&\qquad+(\nu_j-\overline{\nu})(\nabla \xi_{j,h}^n,\nabla v_h)
+b^{*}(u_{j}^{n+1},u_{j}^{n+1},v_{h})-b^{*}(\overline{u}_{h}^{n},u_{j,h}^{n+1},v_{h}) \nonumber\\
&\qquad-b^{*}(u_{j,h}^{n}-\overline{u}_{h}^n,u_{j,h}
^{n},v_{h})-(p_{j}^{n+1},\nabla\cdot v_{h})\nonumber\\
&\qquad=-(\frac{\eta_{j}^{n+1}-\eta_{j}^{n}}{\Delta t},v_{h}) -\overline{\nu}(\nabla\eta
_{j}^{n+1},\nabla v_{h})-(\nu_j-\overline{\nu})(\nabla\eta
_{j}^{n},\nabla v_{h}) 
+\text{Intp}(u_j^{n+1};v_{h}).\nonumber
\end{align*}
Setting $v_{h}=\xi_{j,h}^{n+1}\in V_{h}$ and rearranging the nonlinear
terms, we have
\begin{equation}
\label{eq:err1}
\begin{split}
&\frac{1}{\Delta t}\left(\frac{1}{2}||\xi_{j,h}^{n+1}||^{2}-\frac{1}{2}||\xi
_{j,h}^{n}||^{2}+\frac{1}{2}\|\xi_{j,h}^{n+1}-\xi_{j,h}^{n}\|^{2}\right)+\overline{\nu}
||\nabla\xi_{j,h}^{n+1}||^{2}\\
&\qquad=-(\nu_j-\overline{\nu})(\nabla (u_{j}^{n+1}-u_{j}^n),\nabla \xi_{j,h}^{n+1})
-(\nu_j-\overline{\nu})(\nabla \xi_{j,h}^n,\nabla \xi_{j,h}^{n+1})\\
&\qquad\quad-\overline{\nu}
(\nabla\eta_{j}^{n+1},\nabla\xi_{j,h}^{n+1})-(\nu_j-\overline{\nu})(\nabla\eta
_{j}^{n},\nabla \xi_{j,h}^{n+1})\\
&\qquad\quad-b^{*}(u_{j,h}^n-\overline{u}_h^n,u_{j,h}^{n+1}-u_{j,h}^{n},\xi_{j,h}^{n+1})\\
&\qquad\quad-b^{*}(u_{j}^{n+1},u_{j}^{n+1},\xi_{j,h}^{n+1})+b^{*}(u_{j,h}^{n}%
,u_{j,h}^{n+1},\xi_{j,h}^{n+1})+(p_{j}
^{n+1},\nabla\cdot\xi_{j,h}^{n+1})\\
&\qquad\quad-(\frac{\eta_{j}^{n+1}-\eta_{j}^{n}}{ \Delta t},\xi_{j,h}^{n+1}) 
+\text{Intp}(u_j^{n+1};\xi_{j,h}^{n+1}).
\end{split}
\end{equation}
We first bound the viscous terms on the RHS of \eqref{eq:err1}:
\begin{equation}\label{firstineq}
\begin{aligned}
-(\nu_j-\overline{\nu})& (\nabla (u_{j}^{n+1}-u_j^n), \nabla\xi_{j,h}^{n+1})\leq \vert \nu_j-\overline{\nu}\vert \|\nabla (u_j^{n+1}-u_j^n)\| \|\nabla\xi_{j,h}^{n+1}\| \\
&\leq  \frac{1}{4C_0}\frac{\vert \nu_j-\overline{\nu}\vert^2}{\overline{\nu}}\|\nabla(u_{j}^{n+1}-u_j^n)\|^{2} 
+ C_0\overline{\nu}\|\nabla\xi_{j,h}^{n+1}\|^{2}  \\
&\leq \frac{\Delta t}{4C_0}\frac{\vert \nu_j-\overline{\nu}\vert^2}{\overline{\nu}} \left(\int_{t^{n}}^{t^{n+1}}\| \nabla u_{j,t}\|^{2}\, dt\right)+ C_0\overline{\nu}\|\nabla\xi_{j,h}^{n+1}\|^{2}
\end{aligned}
\end{equation}
\begin{equation}
\begin{aligned}
-\overline{\nu}(\nabla\eta_{j}^{n+1},\nabla\xi_{j,h}^{n+1}) &\leq\overline{\nu}\|\nabla\eta_{j}
^{n+1}\| \|\nabla\xi_{j,h}^{n+1}\| \\
&\leq \frac{\overline{\nu}}{4C_0}\|\nabla\eta_{j}^{n+1}\|^{2}+ C_0 \overline{\nu}\|\nabla\xi_{j,h}%
^{n+1}\|^{2} 
\end{aligned}
\end{equation}
\begin{equation}
\begin{aligned}
-(\nu_j-\overline{\nu})(\nabla\eta_{j}^{n},\nabla\xi_{j,h}^{n+1}) &\leq\vert \nu_j-\overline{\nu}\vert \|\nabla\eta_{j}
^{n+1}\| \|\nabla\xi_{j,h}^{n+1}\| \\
&\leq \frac{1}{4C_0}\frac{\vert \nu_j-\overline{\nu}\vert^2}{\overline{\nu}}\|\nabla\eta_{j}^{n+1}\|^{2}
+ C_0\overline{\nu}\|\nabla\xi_{j,h}^{n+1}\|^{2} 
\end{aligned}
\end{equation}
\begin{equation}\label{ineq:err2}
\begin{aligned}
-(\nu_j-\overline{\nu})(\nabla\xi_{j,h}^{n},\nabla\xi_{j,h}^{n+1}) &\leq\vert \nu_j-\overline{\nu}\vert \|\nabla\xi_{j,h}
^{n}\| \|\nabla\xi_{j,h}^{n+1}\| \\
&\leq \frac{1}{4C_1}\frac{\vert \nu_j-\overline{\nu}\vert^2}{\overline{\nu}} \|\nabla\xi_{j,h}^{n}\|^{2}+ C_1\overline{\nu}\|\nabla\xi_{j,h}%
^{n+1}\|^{2}  \\
&\leq 
\frac{\vert \nu_j-\overline{\nu}\vert}{2} \|\nabla\xi_{j,h}^{n}\|^{2}+ \frac{\vert \nu_j-\overline{\nu}\vert}{2} \|\nabla\xi_{j,h}^{n+1}\|^{2} 
\end{aligned}
\end{equation}
in which we note that both terms on the RHS of \eqref{ineq:err2} need to be hidden in the LHS of the error equation, thus $C_1= \frac{|\nu_j-\overline{\nu}|}{2\nu}$ is selected to minimize the summation.  

Next we analyze the nonlinear terms on the RHS of \eqref{eq:err1} one by one. 
For the first nonlinear term, we have 
\begin{equation}\label{ineq:err3}
\begin{aligned}
&-b^*(u_{j,h}^n-\overline{u}_h^n, u_{j,h}^{n+1}-u_{j,h}^n, \xi_{j,h}^{n+1})\\
=&-b^*(u_{j,h}^n-\overline{u}_h^n, e_{j}^{n+1}-e_{j}^n, \xi_{j,h}^{n+1}) +b^*(u_{j,h}^n-\overline{u}_h^n, u_{j}^{n+1}-u_{j}^n, \xi_{j,h}^{n+1})\\
=&-b^*(u_{j,h}^n-\overline{u}_h^n, \eta_{j}^{n+1}, \xi_{j,h}^{n+1})+b^*(u_{j,h}^n-\overline{u}_h^n, \eta_{j}^{n}, \xi_{j,h}^{n+1})\\
&\quad+b^*(u_{j,h}^n-\overline{u}_h^n, \xi_{j}^{n}, \xi_{j,h}^{n+1})+b^*(u_{j,h}^n-\overline{u}_h^n, u_{j}^{n+1}-u_{j}^n, \xi_{j,h}^{n+1})\,.
\end{aligned}
\end{equation}
Using inequality \eqref{In1} and Young's inequality, we have the estimates
\begin{equation}
\begin{aligned}
-b^*(u_{j,h}^n-\overline{u}_h^n, \eta_{j}^{n+1}&, \xi_{j,h}^{n+1}) \leq C\|\nabla (u_{j,h}^n-\overline{u}_h^n)\|\|\nabla\eta_{j}^{n+1}\|\|\nabla\xi_{j,h}^{n+1}\|\\
&\leq
\frac{C^2}{4C_0} \overline{\nu}^{-1}\|\nabla (u_{j,h}^n-\overline{u}_h^n)\|^{2}\|\nabla\eta_{j}^{n+1}\|^{2}
+C_0\overline{\nu}\|\nabla\xi_{j,h}^{n+1}\|^{2}
\end{aligned}
\end{equation}
\begin{equation}
\begin{aligned}
-b^*(u_{j,h}^n-\overline{u}_h^n, \eta_{j}^{n}, \xi_{j,h}^{n+1}) &\leq C\|\nabla (u_{j,h}^n-\overline{u}_h^n)\|\|\nabla\eta_{j}^{n}\|\|\nabla\xi_{j,h}^{n+1}\|\\
&\leq
\frac{C^2}{4C_0} \overline{\nu}^{-1}\|\nabla (u_{j,h}^n-\overline{u}_h^n)\|^{2}\|\nabla\eta_{j}^{n}\|^{2}
+ C_0\overline{\nu}\|\nabla\xi_{j,h}^{n+1}\|^{2}.
\end{aligned}
\end{equation}
Because $b^*(\cdot,\cdot,\cdot)$ is skew-symmetric, we have
\begin{align*}
b^{\ast}(u_{j,h}^n-\overline{u}_h^n,\xi_{j,h}^{n},\xi_{j,h}^{n+1})&=b^{\ast}(u_{j,h}^n-\overline{u}_h^n,\xi
_{j,h}^{n}-\xi_{j,h}^{n+1},\xi_{j,h}^{n+1}) \\
&=b^{\ast}(u_{j,h}^n-\overline{u}_h^n,\xi_{j,h}^{n+1},\xi_{j,h}^{n+1}-\xi_{j,h}^n)\,.
\end{align*}
Then, by inequality \eqref{In1}, we obtain
\begin{equation}\label{inver}
\begin{aligned}
 b^{\ast}(u_{j,h}^n-&\overline{u}_h^n,\xi_{j,h}^{n},\xi_{j,h}^{n+1})  \\
\leq & \Vert \nabla (u_{j,h}^n-\overline{u}_h^n)\Vert\Vert \nabla \xi_{j,h}^{n}\Vert
\Vert \nabla (\xi_{j,h}^{n+1}-\xi_{j,h}^n)\Vert^{1/2}\Vert  \xi_{j,h}^{n+1}-\xi_{j,h}^n\Vert^{1/2}\\
\leq & C\Vert \nabla (u_{j,h}^n-\overline{u}_h^n)\Vert\Vert \nabla \xi_{j,h}^{n}\Vert (h)^{-1/2}\Vert  \xi_{j,h}^{n+1}-\xi_{j,h}^n\Vert\\
\leq & \frac{1}{4\triangle t}\Vert\xi_{j,h}^{n+1}-\xi_{j,h}^{n}\Vert^{2}+\left(
C\frac{\triangle t}{h}\Vert\nabla (u_{j,h}^n-\overline{u}_h^n)\Vert^{2}\right)  \Vert\nabla
\xi_{j,h}^{n}\Vert^{2}.
\end{aligned}
\end{equation}
For the last term of \eqref{ineq:err3}, we have
\begin{align}
b^{*}(u_{j,h}^n-\overline{u}_h^n,&u_{j}^{n+1}-u_{j}^{n},\xi_{j,h}^{n+1}) \leq C\|\nabla
(u_{j,h}^n-\overline{u}_h^n)\|\|\nabla(u_{j}^{n+1}-u_{j}^{n})\|\|\nabla\xi_{j,h}^{n+1}%
\|\nonumber\\
&\leq 
\frac{C^2}{4C_0}\overline{\nu}^{-1}\|\nabla (u_{j,h}^n-\overline{u}_h^n)\|^{2}\|\nabla (u_{j}^{n+1}-u_{j}^{n})\|^{2} 
+C_0\overline{\nu}\|\nabla\xi_{j,h}^{n+1}\|^{2} \\
&\leq
\frac{C^2 \Delta t}{4C_0}\overline{\nu}^{-1}\|\nabla (u_{j,h}^n-\overline{u}_h^n)\|^{2}\left(\int_{t^{n}}^{t^{n+1}}\| \nabla u_{j,t}\|^{2} \, dt\right) 
+C_0\overline{\nu} \|\nabla\xi_{j,h}^{n+1}\|^{2}
.\nonumber
\end{align}
Next, we bound the last two nonlinear terms on the RHS of \eqref{eq:err1} as follows:
\begin{equation}\label{eq:nonlinear}
\begin{aligned}
-&b^{*}(u_{j}^{n+1}, u_{j}^{n+1},\xi_{j,h}^{n+1})
+b^{*}(u_{j,h}^{n},u_{j,h}^{n+1},\xi_{j,h}^{n+1})\\
& = -b^{*}(e_{j}^{n},u_{j}^{n+1},\xi_{j,h}^{n+1})
-b^{*}(u_{j,h}^{n},e_{j}^{n+1},\xi_{j,h}^{n+1})
-b^{*}(u^{n+1}_{j}-u_j^{n}, u_{j}^{n+1}, \xi_{j,h}^{n+1}) \\
&= -b^{*}(\eta^{n}_j,u_{j}^{n+1},\xi_{j,h}^{n+1})-b^{*}(\xi_{j,h}^{n},u_{j}^{n+1},\xi_{j,h}^{n+1})\\
&\qquad -b^{*}(u^{n}_{j,h}, \eta_{j}^{n+1}, \xi_{j,h}^{n+1})-b^{*}(u^{n+1}_{j}-u^{n}_j,
u_{j}^{n+1}, \xi_{j,h}^{n+1}) ,
\end{aligned}
\end{equation}
where, with the assumption $u_j^{n+1}\in L^{\infty}(0,T; H^1(\Omega))$, we have 
\begin{equation}
\begin{aligned}
-b^{*}(\eta_{j}^{n},u_{j}^{n+1},\xi_{j,h}^{n+1}) 
&\leq 
C\|\nabla \eta_{j}^{n}\|\|\nabla u_{j}^{n+1}\|\|\nabla\xi_{j,h}^{n+1}\|  \\
&\leq
\frac{C^2}{4C_0}\overline{\nu}^{-1}\|\nabla\eta_{j}^{n}\|^{2} + C_0\overline{\nu}\|\nabla\xi_{j,h}^{n+1}\|^{2}.
\end{aligned}
\end{equation}
Using the inequality \eqref{In1}, Young's inequality, and $u_j^{n+1}\in L^{\infty}(0,T; H^1(\Omega))$, we get
\begin{align}
-b^{*}(\xi^{n}_{j,h}, u_{j}^{n+1}, \xi_{j,h}^{n+1}) &\leq C\| \nabla\xi
^{n}_{j,h} \|^{1/2} \Vert\xi^{n}_{j,h}\Vert^{1/2}\|\nabla
u_{j}^{n+1}\|\|\nabla\xi_{j,h}^{n+1}\|\nonumber\\
&\leq C\| \nabla\xi^{n}_{j,h} \|^{1/2} \Vert\xi^{n}_{j,h}\Vert
^{1/2}\|\nabla\xi_{j,h}^{n+1}\|\nonumber\\
&\leq 
C\Big(\frac{1}{4\alpha}\|\nabla\xi^{n}_{j,h} \|\|\xi^{n}_{j,h} \| + \alpha\|\nabla\xi_{j,h}^{n+1}\|^{2}\Big)\\
&\leq 
C\Big(\frac{1}{4\alpha} \big(\frac{\delta}{2}\|\nabla\xi^{n}_{j,h} \|^{2}+\frac{1}{2\delta}\|\xi^{n}_{j,h} \|^2 \big)
+ \alpha\|\nabla\xi_{j,h}^{n+1}\|^{2} \Big)\nonumber\\
&\leq
C_0\overline{\nu} \|\nabla \xi^{n}_{j,h} \|^{2}+\frac{C^4}{64C_0^3\overline{\nu}^{3}}\|\xi^{n}_{j,h} \|^{2} + C_0 \overline{\nu} \|\nabla\xi_{j,h}^{n+1}\|^{2}\, ,\nonumber
\end{align}
where we set $\alpha= \frac{C_0\overline{\nu}}{C}$ and $\delta= \frac{8C_0^2\overline{\nu}^2}{C^2}$. 
By Young's inequality, \eqref{In1}, and the result \eqref{eq:sta} from the stability analysis, i.e., $\Vert u_{j,h}^n \Vert^2 \leq C$, we also have
\begin{equation}
\begin{aligned}
b^{*}(u^{n}_{j,h}, \eta_{j}^{n+1}, \xi_{j,h}^{n+1}) 
&\leq 
C\|\nabla u^{n}_{j,h}\|^{1/2}\Vert u_{j,h}^n\Vert^{1/2}\|\nabla \eta_{j}^{n+1}\|\|\nabla\xi_{j,h}^{n+1}\|\\
&\leq
\frac{C^2}{4C_0}\overline{\nu}^{-1}\|\nabla u_{j,h}^{n}\|\|\nabla\eta^{n+1}_{j}\|^{2} + C_0\overline{\nu}\|\nabla\xi_{j,h}^{n+1}\|^{2}
\end{aligned}
\end{equation}
and 
\begin{align}
b^{*}(u_{j}^{n+1}-u_{j}^{n},u_{j}^{n+1},\xi_{j,h}^{n+1}) 
&\leq 
 C\|\nabla (u_{j}^{n+1}-u_{j}^{n})\|\|\nabla u_{j}^{n+1}\|\|\nabla\xi_{j,h}%
^{n+1}\|\nonumber\\
\leq & \frac{C^2}{4C_0}\overline{\nu}^{-1}\|\nabla(u_{j}^{n+1}-u_{j}^{n})\|^{2} + C_0 \overline{\nu} \|\nabla\xi_{j,h}^{n+1}\|^{2} \nonumber\\
= & \frac{C^2\Delta t^2}{4C_0}\overline{\nu}^{-1}\left\|\frac{\nabla u_{j}^{n+1}- \nabla u_{j}^{n}}{\Delta t}\right\|^{2} + C_0 \overline{\nu} \|\nabla\xi_{j,h}^{n+1}\|^{2} \\
= &\frac{C^2\Delta t^2}{4C_0}\overline{\nu}^{-1} 
\int_{\Omega}\left(\frac{1}{\Delta t}\int_{t^{n}}^{t^{n+1}} \nabla u_{j,t}
\, dt\right)^{2} d\Omega+ C_0 \overline{\nu} \|\nabla\xi_{j,h}^{n+1}\|^{2} \nonumber\\
%\leq & \frac{C^2\Delta t^2}{4C_0}\overline{\nu}^{-1} 
%\int_{\Omega}\frac{1}{\Delta t}\int_{t^{n}}^{t^{n+1}} |\nabla u_{j,t}|^2
%\, dt\, d\Omega + C_0 \overline{\nu} \|\nabla\xi_{j,h}^{n+1}\|^{2} \nonumber\\
\leq & \frac{C^2 \Delta t}{4C_0} \overline{\nu}^{-1} \int_{t^{n}}^{t^{n+1}}\| \nabla u_{j,t} \|^{2}\, dt
 + C_0 \overline{\nu} \|\nabla\xi_{j,h}^{n+1}\|^{2}. \nonumber
\end{align}
For the pressure term in \eqref{eq:err1}, because $\xi_{j,h}^{n+1}\in V_{h}$, we have%
\begin{align}
(p_{j}^{n+1},\nabla\cdot\xi_{j,h}^{n+1})&=(p_{j}^{n+1}-q_{j,h}^{n+1},
\nabla\cdot\xi_{j,h}^{n+1})\nonumber\\
&\leq\sqrt{d}\,\|p_{j}^{n+1}-q_{j,h}^{n+1}\|\|\nabla\xi_{j,h}^{n+1}\|\\
&\leq
\frac{1}{4d\, C_0} \overline{\nu}^{-1}\|p_{j}^{n+1}-q_{j,h}^{n+1}\|^{2} + C_0\,\overline{\nu}\|\nabla\xi_{j,h}^{n+1}\|^{2}
 \text{ .}\nonumber
\end{align}
The other terms are bounded as
\begin{align}
\Big( \frac{\eta_{j}^{n+1}-\eta_{j}^{n}}{ \Delta t},\xi_{j,h}^{n+1} \Big) 
&\leq
C \Big\|\frac{\eta_{j}^{n+1}-\eta_{j}^{n}}{ \Delta t} \Big\| \|\nabla\xi_{j,h}%
^{n+1}\|\nonumber\\
&\leq 
\frac{C^2}{4C_0} \overline{\nu}^{-1}\left \|\frac{\eta_{j}^{n+1}-\eta_{j}^{n}}{ \Delta t} \right\|^{2}
+C_0 \overline{\nu}\|\nabla\xi_{j,h}^{n+1}\|^{2}\\
&\leq 
\frac{C^2}{4C_0} \overline{\nu}^{-1}\left \|\frac{1}{\Delta t}\int_{t^{n}}^{t^{n+1}} \eta_{j,t} \text{ }
dt \right \|^2+C_0 \overline{\nu}\|\nabla\xi_{j,h}^{n+1}\|^{2}\nonumber\\
&\leq
\frac{1}{4C_0}\frac{C^2}{\overline{\nu}\Delta t}\int_{t^{n}}^{t^{n+1}}\| \eta_{j,t}\|^{2}\text{ }
dt+C_0 \overline{\nu}\|\nabla\xi_{j,h}^{n+1}\|^{2}\nonumber
\end{align}
and
\begin{align}
\text{Intp}(u_{j}^{n+1};\xi_{j,h}^{n+1})
&=\left(\frac{u_{j}^{n+1}-u_{j}^{n}}{\Delta
t}-u_{j,t}(t^{n+1}),\xi_{j,h}^{n+1}\right)\nonumber\\
&\leq C\left\|\frac{u_{j}^{n+1}-u_{j}^{n}}{\Delta t}-u_{j,t}(t^{n+1})\right\| \|\nabla
\xi_{j,h}^{n+1}\|\nonumber\\
&\leq
\frac{C^2}{4C_0}\overline{\nu}^{-1}\left \|\frac{u_{j}^{n+1}-u_{j}^{n}}{\Delta t}-u_{j,t}(t^{n+1}) \right\|^{2}
+ C_0\overline{\nu}\|\nabla\xi_{j,h}^{n+1}\|^{2} \nonumber \\
&\leq
\frac{C^2\Delta t}{4C_0}\overline{\nu}^{-1} \int_{t^{n}}^{t^{n+1}}\|u_{j,tt}\|^{2}\, dt 
+ C_0\overline{\nu}\|\nabla\xi_{j,h}^{n+1}\|^{2}
.
\label{lastineq}
\end{align}

\noindent Combining \eqref{firstineq}-\eqref{lastineq}, we have
\begin{align}
&\frac{1}{\Delta t}\left(\frac{1}{2}||\xi_{j,h}^{n+1}||^{2}-\frac{1}{2}||\xi
_{j,h}^{n}||^{2}+\frac{1}{4}\Vert\xi_{j,h}^{n+1}-\xi_{j,h}^{n}\Vert^{2} \right)+ C_0 \overline{\nu} \Vert \nabla \xi_{j,h}^{n+1}\Vert^2\nonumber\\
&+ \overline{\nu} \left(1-15C_0-\frac{\vert \nu_j-\overline{\nu}\vert}{2 \overline{\nu}}\right)\left(\Vert\nabla \xi_{j,h}^{n+1}\Vert^{2}-\Vert\nabla \xi_{j,h}^{n}\Vert
^{2}\right)\nonumber\\
&+\overline{\nu} \left((1-\sigma)\left(1-15C_0-\frac{\vert \nu_j-\overline{\nu}\vert}{2 \overline{\nu}}\right)-\frac{C \Delta
t}{\overline{\nu} h}\Vert\nabla(u_{j,h}^{n}-\overline{u}_{h}^{n})\Vert^{2}\right) \Vert\nabla
\xi_{j,h}^{n}\Vert^{2}\nonumber\\
&+\overline{\nu}  \left(\sigma\left(1-15C_0-\frac{\vert \nu_j-\overline{\nu}\vert}{2 \overline{\nu}}\right)-\frac{\vert \nu_j-\overline{\nu}\vert}{2 \overline{\nu}}\right)\Vert \nabla \xi_{j,h}^n\Vert^2\nonumber\\
&\leq \frac{C}{\overline{\nu}^{3}}\Vert\xi_{j,h}^{n}\Vert^{2}
+C\Delta t\frac{\vert \nu_j-\overline{\nu}\vert^2}{\overline{\nu}} \int_{t^{n}}^{t^{n+1}}\| \nabla u_{j,t}\|^{2} \text{
}dt
+C\overline{\nu}\Vert\nabla\eta_{j}^{n+1}\Vert^{2}\label{eq:err2} \\
&\quad+ C\frac{\vert \nu_j-\overline{\nu}\vert^2}{\overline{\nu}}\Vert\nabla\eta_{j}^{n+1}\Vert^{2}
+C\overline{\nu}^{-1}\Vert\nabla
(u_{j,h}^n-\overline{u}_{h}^n)\Vert^{2}
\Vert\nabla\eta_{j}^{n+1}\Vert^{2} \nonumber\\
&\quad+ C\overline{\nu}^{-1}\Vert\nabla (u_{j,h}^n-\overline{u}_h^n)\Vert^{2}\Vert\nabla\eta_{j}^{n}\Vert^{2}
+C\overline{\nu}^{-1} \Delta t \Vert\nabla (u_{j,h}^n-\overline{u}_h^n)\Vert^{2} \int_{t^{n}}^{t^{n+1}%
}\Vert\nabla u_{j,t}\Vert^{2}\, dt \nonumber\\
&\quad+ C\overline{\nu}^{-1}\Vert\nabla\eta_{j}^{n}\Vert^{2}
+C\overline{\nu}^{-1}\Vert\nabla u_{j,h}^{n}\Vert^{2}\Vert\nabla\eta_{j}^{n+1}\Vert^{2}
+C\overline{\nu}^{-1}\Delta t \int_{t^{n}}^{t^{n+1}}\Vert\nabla u_{j,t}\Vert^{2}\, dt \nonumber\\
&\quad+ C\overline{\nu}^{-1}\Vert p_{j}^{n+1}-q_{j,h}^{n+1}\Vert^{2}
+C\overline{\nu}^{-1}\Delta t^{-1}\int_{t^{n}}^{t^{n+1}}\Vert\eta_{j,t}\Vert^{2}\, dt
+C\overline{\nu}^{-1} \Delta t \int_{t^{n}}^{t^{n+1}}\Vert u_{j,tt}\Vert^{2}dt. \nonumber
\end{align}
Note that the generic constant $C$ independent of $\Delta t$  is used on the RHS. It, however, depends on the geometry and mesh due to the use of inverse inequality in \eqref{inver}.

Similar to the stability analysis, we take $C_0= \frac{1}{15}(1-\frac{\sigma+1}{2\sigma}\sqrt{\mu}) =  \frac{1}{15} \frac{\epsilon} {\sqrt{\mu}+\epsilon}( 1-\frac{\sqrt{\mu}}{2}) $ with $\sigma=\frac{\sqrt{\mu}+\epsilon}{2-\sqrt{\mu}}$. Then, \eqref{eq:err2} becomes
\begin{align}
&\frac{1}{\Delta t}\left(\frac{1}{2}||\xi_{j,h}^{n+1}||^{2}-\frac{1}{2}||\xi
_{j,h}^{n}||^{2}+\frac{1}{4}\Vert\xi_{j,h}^{n+1}-\xi_{j,h}^{n}\Vert^{2}%
\right)+\frac{1}{15}\frac{\epsilon}{\sqrt{\mu}+\epsilon}(1-\frac{\sqrt{\mu}}{2}) \overline{\nu} \Vert \nabla \xi_{j,h}^{n+1}\Vert^2\nonumber\\
&+ \overline{\nu} \left(\frac{\sqrt{\mu}}{2}\frac{(2+\epsilon)}{\sqrt{\mu}+\epsilon}-\frac{\vert \nu_j-\overline{\nu}\vert}{2 \overline{\nu}}\right)\left(\Vert\nabla \xi_{j,h}^{n+1}\Vert^{2}-\Vert\nabla \xi_{j,h}^{n}\Vert
^{2}\right)\nonumber\\
&+\overline{\nu} \left(\frac{2-2\sqrt{\mu}-\epsilon}{2-\sqrt{\mu}}\left(\frac{\sqrt{\mu}}{2}\frac{2+\epsilon}{\sqrt{\mu}+\epsilon}-\frac{\vert \nu_j-\overline{\nu}\vert}{2 \overline{\nu}}\right)-\frac{C \Delta
t}{\overline{\nu} h}\Vert\nabla(u_{j,h}^{n}-\overline{u}_{h}^{n})\Vert^{2}\right) \Vert\nabla
\xi_{j,h}^{n}\Vert^{2}\nonumber\\
&+\overline{\nu}  \left(\frac{\sqrt{\mu}+\epsilon}{2-\sqrt{\mu}}\left(\frac{\sqrt{\mu}}{2}\frac{2+\epsilon}{\sqrt{\mu}+\epsilon}-\frac{\vert \nu_j-\overline{\nu}\vert}{2 \overline{\nu}}\right)-\frac{\vert \nu_j-\overline{\nu}\vert}{2 \overline{\nu}}\right)\Vert \nabla \xi_{j,h}^n\Vert^2\nonumber\\
&\leq \frac{C}{\overline{\nu}^{3}}\Vert\xi_{j,h}^{n}\Vert^{2}
+C\Delta t\frac{\vert \nu_j-\overline{\nu}\vert^2}{\overline{\nu}} \int_{t^{n}}^{t^{n+1}}\| \nabla u_{j,t}\|^{2} \text{
}dt
+C\overline{\nu}\Vert\nabla\eta_{j}^{n+1}\Vert^{2}\label{eq:err21} \\
&\quad+ C\frac{\vert \nu_j-\overline{\nu}\vert^2}{\overline{\nu}}\Vert\nabla\eta_{j}^{n+1}\Vert^{2}
+C\overline{\nu}^{-1}\Vert\nabla
(u_{j,h}^n-\overline{u}_{h}^n)\Vert^{2}
\Vert\nabla\eta_{j}^{n+1}\Vert^{2} \nonumber\\
&\quad+ C\overline{\nu}^{-1}\Vert\nabla (u_{j,h}^n-\overline{u}_h^n)\Vert^{2}\Vert\nabla\eta_{j}^{n}\Vert^{2}
+C\overline{\nu}^{-1} \Delta t \Vert\nabla (u_{j,h}^n-\overline{u}_h^n)\Vert^{2} \int_{t^{n}}^{t^{n+1}%
}\Vert\nabla u_{j,t}\Vert^{2}\, dt \nonumber\\
&\quad+ C\overline{\nu}^{-1}\Vert\nabla\eta_{j}^{n}\Vert^{2}
+C\overline{\nu}^{-1}\Vert\nabla u_{j,h}^{n}\Vert^{2}\Vert\nabla\eta_{j}^{n+1}\Vert^{2}
+C\overline{\nu}^{-1}\Delta t \int_{t^{n}}^{t^{n+1}}\Vert\nabla u_{j,t}\Vert^{2}\, dt \nonumber\\
&\quad+ C\overline{\nu}^{-1}\Vert p_{j}^{n+1}-q_{j,h}^{n+1}\Vert^{2}
+C\overline{\nu}^{-1}\Delta t^{-1}\int_{t^{n}}^{t^{n+1}}\Vert\eta_{j,t}\Vert^{2}\, dt
+C\overline{\nu}^{-1} \Delta t \int_{t^{n}}^{t^{n+1}}\Vert u_{j,tt}\Vert^{2}dt\text{ .} \nonumber
\end{align}
By the convergence condition \eqref{conv2}, we have 
\begin{align*}
\frac{\sqrt{\mu}}{2}\frac{(2+\epsilon)}{\sqrt{\mu}+\epsilon}-\frac{\vert \nu_j-\overline{\nu}\vert}{2 \overline{\nu}}&\geq \left(\frac{\sqrt{\mu}}{2}\frac{(2+\epsilon)}{\sqrt{\mu}+\epsilon}-\frac{\sqrt{\mu}}{2}\right)\geq \frac{\sqrt{\mu}}{2}\frac{2-\sqrt{\mu}}{\sqrt{\mu}+\epsilon} > 0,
\end{align*}
and 
\begin{align*}
&\frac{\sqrt{\mu}+\epsilon}{2-\sqrt{\mu}}\left(\frac{\sqrt{\mu}}{2}\frac{2+\epsilon}{\sqrt{\mu}+\epsilon}-\frac{\vert \nu_j-\overline{\nu}\vert}{2 \overline{\nu}}\right)-\frac{\vert \nu_j-\overline{\nu}\vert}{2 \overline{\nu}}\\
&\qquad\qquad\geq  \frac{\sqrt{\mu}+\epsilon}{2-\sqrt{\mu}}\left(\frac{\sqrt{\mu}}{2}\frac{2+\epsilon}{\sqrt{\mu}+\epsilon}-\frac{\sqrt{\mu}}{2}\right)-\frac{\sqrt{\mu}}{2}\\
&\qquad\qquad\geq  \frac{\sqrt{\mu}+\epsilon}{2-\sqrt{\mu}}\left(\frac{\sqrt{\mu}}{2}\frac{2-\sqrt{\mu}}{\sqrt{\mu}+\epsilon}\right)-\frac{\sqrt{\mu}}{2} \geq \frac{\sqrt{\mu}}{2}-\frac{\sqrt{\mu}}{2}=0
\end{align*}
and by the convergence conditions \eqref{conv1} and \eqref{conv2}, we have
\begin{align*}
&\frac{2-2\sqrt{\mu}-\epsilon}{2-\sqrt{\mu}}\left(\frac{\sqrt{\mu}}{2}\frac{2+\epsilon}{\sqrt{\mu}+\epsilon}-\frac{\vert \nu_j-\overline{\nu}\vert}{2 \overline{\nu}}\right)-\frac{C \Delta
t}{\overline{\nu} h}\Vert\nabla(u_{j,h}^{n}-
\overline{u}_{h}^{n})\Vert^{2}\\
&\qquad\qquad\geq\frac{2-2\sqrt{\mu}-\epsilon}{2-\sqrt{\mu}}\left(\frac{\sqrt{\mu}}{2}\frac{2-\sqrt{\mu}}{\sqrt{\mu}+\epsilon}\right)-\frac{(2-2\sqrt{\mu}-\epsilon)\sqrt{\mu}}{2(\sqrt{\mu}+\epsilon)}\\
&\qquad\qquad\geq \frac{(2-2\sqrt{\mu}-\epsilon)\sqrt{\mu}}{2(\sqrt{\mu}+\epsilon)}-\frac{(2-2\sqrt{\mu}-\epsilon)\sqrt{\mu}}{2(\sqrt{\mu}+\epsilon)}=0.
\end{align*}
Summing (\ref{eq:err2}) from $n=1$ to $N-1$ and multiplying both sides by $\Delta t$ gives
\begin{align*}
&\frac{1}{2}\|\xi_{j,h}^{N}\|^{2}+\frac{1}{4}\sum_{n=0}^{N-1}\Vert\xi_{j,h}^{n+1}
-\xi_{j,h}^{n}\Vert^{2}+\frac{1}{15}\frac{\epsilon}{\sqrt{\mu}+\epsilon}(1-\frac{\sqrt{\mu}}{2})\overline{\nu}\Delta t\sum_{n=0}^{N-1}\|\nabla\xi_{j,h}^{n+1}\|^{2}
\\
&\qquad+\overline{\nu} \Delta t \left(\frac{\sqrt{\mu}}{2}\frac{(2+\epsilon)}{\sqrt{\mu}+\epsilon}-\frac{\vert \nu_j-\overline{\nu}\vert}{2 \overline{\nu}}\right)\|\nabla\xi_{j,h}^{N}\|^{2}\\
&\leq\frac{1}{2}\|\xi_{j,h}^{0}\|^{2}+\overline{\nu} \Delta t \left(\frac{\sqrt{\mu}}{2}\frac{(2+\epsilon)}{\sqrt{\mu}+\epsilon}-\frac{\vert \nu_j-\overline{\nu}\vert}{2 \overline{\nu}}\right)\|\nabla\xi_{j,h}^{0}\|^{2}
+\frac{C\Delta t}{\overline{\nu}^{3}}\sum_{n=0}^{N-1}\Vert\xi_{j,h}^{n}\Vert^{2}\\
&\quad+\Delta t\sum_{n=0}^{N-1}
\bigg\{
C\Delta t\frac{\vert \nu_j-\overline{\nu}\vert^2}{\overline{\nu}} \int_{t^{n}}^{t^{n+1}}\| \nabla u_{j,t}\|^{2} \,dt
+C\overline{\nu}\Vert\nabla\eta_{j}^{n+1}\Vert^{2}\\
&\quad+C\frac{\vert \nu_j-\overline{\nu}\vert^2}{\overline{\nu}}\Vert\nabla\eta_{j}^{n+1}\Vert^{2}
+C\overline{\nu}^{-1}\Vert\nabla
(u_{j,h}^n-\overline{u}_{h}^n)\Vert^{2}
\Vert\nabla\eta_{j}^{n+1}\Vert^{2}\\
&\quad+C\overline{\nu}^{-1}\Vert\nabla (u_{j,h}^n-\overline{u}_h^n)\Vert^{2}\Vert\nabla\eta_{j}^{n}\Vert^{2}
+C\overline{\nu}^{-1}\Delta t \Vert\nabla (u_{j,h}^n-\overline{u}_h^n)\Vert^{2}
\int_{t^{n}}^{t^{n+1}}\Vert\nabla u_{j,t}\Vert^{2}\, dt \\ 
&\quad+C\overline{\nu}^{-1}\Vert\nabla\eta_{j}^{n}\Vert^{2}
+C\overline{\nu}^{-1}\Delta t \int_{t^{n}}^{t^{n+1}}\Vert\nabla u_{j,t}\Vert^{2}\,dt
+C\overline{\nu}^{-1}\Vert\nabla\eta_{j}^{n+1}\Vert^{2}\Vert\nabla u_{j,h}^{n+1}\Vert \\
&\,+C\overline{\nu}^{-1}\Vert p_{j}^{n+1}-q_{j,h}^{n+1}\Vert^{2}
+C \overline{\nu}^{-1} \Delta t^{-1}\int_{t^{n}}^{t^{n+1}}\Vert\eta_{j,t}\Vert^{2}\, dt
+C\overline{\nu}^{-1} \Delta t \int_{t^{n}}^{t^{n+1}}\Vert u_{j,tt}\Vert^{2}\, dt 
\bigg\} \text{ .}
\end{align*}
\label{eq:err3}
Using the interpolation inequality \eqref{interp2} and  the result \eqref{eq:sta} from the stability analysis, i.e., $\Delta t \sum_{n=0}^{N-1}\Vert \nabla u_{j,h}^{n+1} \Vert^2 \leq C$, we have
\begin{align}
C\overline{\nu}^{-1}\Delta t\sum_{n=0}^{N-1}\Vert\nabla\eta_{j}^{n+1}\Vert^{2}&\Vert\nabla u_{j,h}^{n+1}\Vert 
\leq C\overline{\nu}^{-1}h^{2k}\Delta t \sum_{n=0}^{N-1}\Vert u_j^{n+1}\Vert^2_{k+1}\Vert^{2}\Vert\nabla u_{j,h}^{n+1}\Vert\label{ineq:err4}\\
\leq & C\nu^{-1}h^{2k}\left( \Delta t \sum_{n=0}^{N-1}\Vert u_{j}^{n+1}\Vert_{k+1}^4+\Delta t \sum_{n=0}^{N-1}\Vert\nabla u_{j,h}^{n+1}\Vert^2 \right)\\
\leq & C\nu^{-1}h^{2k}\vertiii{ u_j}^4_{4, k+1}+C\nu^{-1}h^{2k}.\nonumber
\end{align}
Applying the interpolation inequalities \eqref{Interp1}, \eqref{interp2}, and \eqref{interp3} gives
\begin{align}
&\frac{1}{2}\|\xi_{j,h}^{N}\|^{2}+ \overline{\nu}\Delta t \left(\frac{\sqrt{\mu}}{2}\frac{(2+\epsilon)}{\sqrt{\mu}+\epsilon}-\frac{\vert \nu_j-\overline{\nu}\vert}{2 \overline{\nu}}\right)\|\nabla\xi_{j,h}^{N}\|^{2}
+\sum_{n=0}^{N-1}\frac{1}{4}\Vert\xi_{j,h}^{n+1}-\xi_{j,h}^{n}\Vert^{2}
\nonumber\\
&\qquad +\frac{1}{15}\frac{\epsilon}{\sqrt{\mu}+\epsilon}(1-\frac{\sqrt{\mu}}{2})\overline{\nu}\Delta t\sum_{n=0}^{N-1}\|\nabla\xi_{j,h}^{n+1}\|^{2}\nonumber\\
&\leq\frac{1}{2}\|\xi_{j,h}^{0}\|^{2}+\overline{\nu} \Delta t \left(\frac{\sqrt{\mu}}{2}\frac{(2+\epsilon)}{\sqrt{\mu}+\epsilon}-\frac{\vert \nu_j-\overline{\nu}\vert}{2 \overline{\nu}}\right)\|\nabla\xi_{j,h}^{0}\|^{2}
+\frac{C\Delta t}{\overline{\nu}^{3}}\sum_{n=0}^{N-1}\Vert\xi_{j,h}^{n}\Vert^{2}\label{ineq:errlast}
 \\
&\quad+ C\Delta t^2\frac{\vert \nu_j -\overline{\nu}\vert^2}{\overline{\nu}}\vertiii{ \nabla u_{j,t}}^2_{2,0}
+C\overline{\nu}h^{2k}\vertiii{ u_j }^2_{2,k+1}+C\frac{\vert \nu_j - \overline{\nu}\vert^2}{\overline{\nu}}h^{2k}\vertiii{ u_j}^2_{2,k+1}
\nonumber \\
&\quad+Ch^{2k+1}\Delta t^{-1}\vertiii{ u_j}^2_{2,k+1}+C h \Delta t \vertiii{ \nabla u_{j,t}}^2_{2,0}+ C\overline{\nu}^{-1}h^{2k}\vertiii{ u_j}^2_{2,k+1}
\nonumber\\
&\quad  +C\overline{\nu}^{-1}\Delta t^2\vertiii{ \nabla u_{j,t}}^2_{2,0} + C\overline{\nu}^{-1}h^{2k}\vertiii{ u_j}^4_{4, k+1}+C\overline{\nu}^{-1}h^{2k}
\nonumber\\
&\quad+C\overline{\nu}^{-1} h^{2s+2}\Vert|p_{j}|\Vert_{2,s+1}^{2}+C\overline{\nu}^{-1} h^{2k+2}\Vert
|u_{j,t}|\Vert_{2,k+1}^{2}
+C\overline{\nu}^{-1}\Delta t^{2}\Vert|u_{j,tt}|\Vert_{2,0}^{2}.\nonumber
\end{align}
The next step is the application of the discrete Gronwall
inequality (see \cite[p. 176]{GR79}):
\begin{align}
&\frac{1}{2}\|\xi_{j,h}^{N}\|^{2}+ \overline{\nu}\Delta t\left(\frac{\sqrt{\mu}}{2}\frac{(2+\epsilon)}{\sqrt{\mu}+\epsilon}-\frac{\vert \nu_j-\overline{\nu}\vert}{2 \overline{\nu}}\right) \|\nabla\xi_{j,h}^{N}\|^{2}
+\sum_{n=0}^{N-1}\frac{1}{4}\Vert\xi_{j,h}^{n+1}-\xi_{j,h}^{n}\Vert^{2}
\nonumber\\
&\qquad+\frac{1}{15}\frac{\epsilon}{\sqrt{\mu}+\epsilon}(1-\frac{\sqrt{\mu}}{2})\overline{\nu}\Delta t\sum_{n=0}^{N-1}\|\nabla\xi_{j,h}^{n+1}\|^{2}\label{ineq:errlast1}\\
&\leq e^{\frac{CT}{\nu^{3}}}
\bigg\{\frac{1}{2}||\xi_{j,h}^{0}||^{2}
+\overline{\nu}\Delta t \left(\frac{\sqrt{\mu}}{2}\frac{(2+\epsilon)}{\sqrt{\mu}+\epsilon}-\frac{\vert \nu_j-\overline{\nu}\vert}{2 \overline{\nu}}\right)\|\nabla\xi_{j,h}^{0}\|^{2}
\nonumber\\
&\qquad+C\Delta t^2\frac{\vert \nu_j -\overline{\nu}\vert^2}{\overline{\nu}}\vertiii{ \nabla u_{j,t}}^2_{2,0}+C\overline{\nu}h^{2k}\vertiii{ u_j }^2_{2,k+1}
+C\frac{\vert \nu_j - \overline{\nu}\vert^2}{\overline{\nu}}h^{2k}\vertiii{ u_j}^2_{2,k+1}
\nonumber\\
&\qquad+Ch^{2k+1}\Delta t^{-1}\vertiii{ u_j}^2_{2,k+1}+C h \Delta t \vertiii{ \nabla u_{j,t}}^2_{2,0} 
+ C\overline{\nu}^{-1}h^{2k}\vertiii{ u_j}^2_{2,k+1} \nonumber\\
& \qquad+C\overline{\nu}^{-1}\Delta t^2\vertiii{ \nabla u_{j,t}}^2_{2,0}+ C\overline{\nu}^{-1}h^{2k}\vertiii{ u_j}^4_{4, k+1}+C\overline{\nu}^{-1}h^{2k}
 \nonumber \\
&\qquad+C\overline{\nu}^{-1} h^{2s+2}\Vert|p_{j}|\Vert_{2,s+1}^{2}+C\overline{\nu}^{-1} h^{2k+2}\Vert
|u_{j,t}|\Vert_{2,k+1}^{2}
+C\overline{\nu}^{-1}\Delta t^{2}\Vert|u_{j,tt}|\Vert_{2,0}^{2}
\bigg\}.\nonumber
\end{align}

Recall that $e_{j}^{n}=\eta_{j}^{n}+\xi_{j,h}^{n}$. Using the triangle
inequality on the error equation to split the error terms into terms of
$\eta_{j}^{n}$ and $\xi_{j,h}^{n}$ gives
\begin{equation*}
\begin{aligned}
\frac{1}{2}&\Vert e_{j}^{N}\Vert^{2}
+\frac{1}{15}\frac{\epsilon}{\sqrt{\mu}+\epsilon}(1-\frac{\sqrt{\mu}}{2})\overline{\nu} \Delta t\sum_{n=0}^{N-1}\Vert\nabla e_{j}^{n+1}\Vert^{2}
\\
&\leq\frac{1}{2}\Vert\xi_{j,h}^{N}\Vert^{2}
+\frac{1}{15}\frac{\epsilon}{\sqrt{\mu}+\epsilon}(1-\frac{\sqrt{\mu}}{2})\overline{\nu} \Delta t\sum_{n=0}^{N-1}\Vert\nabla \xi_{j,h}^{n+1}\Vert^{2}
 \\
&+\frac{1}{2}\Vert\eta_{j}^{N}\Vert^{2}
+\frac{1}{15}\frac{\epsilon}{\sqrt{\mu}+\epsilon}(1-\frac{\sqrt{\mu}}{2})\overline{\nu} \Delta t\sum_{n=0}^{N-1}\Vert\nabla \eta_{j}^{n+1}\Vert
^{2}
\end{aligned}
\end{equation*}
and 
\begin{equation*}
\begin{aligned}
\frac{1}{2}\Vert\xi_{j,h}^{0}&\Vert^{2}
+\left(\frac{\sqrt{\mu}}{2}\frac{(2+\epsilon)}{\sqrt{\mu}+\epsilon}-\frac{\vert \nu_j-\overline{\nu}\vert}{2 \overline{\nu}}\right)\overline{\nu}\Delta t \Vert\nabla \xi_{j,h}
^{0}\Vert^{2}\\
&\leq \frac{1}{2}\Vert e_{j}^{0}\Vert^{2}
+\left(\frac{\sqrt{\mu}}{2}\frac{(2+\epsilon)}{\sqrt{\mu}+\epsilon}-\frac{\vert \nu_j-\overline{\nu}\vert}{2 \overline{\nu}}\right)\overline{\nu}\Delta t \Vert\nabla e_{j}^{0}\Vert^{2}\\
&\quad+\frac{1}{2}\Vert\eta_{j}^{0}\Vert^{2}
+\left(\frac{\sqrt{\mu}}{2}\frac{(2+\epsilon)}{\sqrt{\mu}+\epsilon}-\frac{\vert \nu_j-\overline{\nu}\vert}{2 \overline{\nu}}\right)\overline{\nu}\Delta t \Vert\nabla \eta_{j}^{0}\Vert^{2}\, .
\end{aligned}
\end{equation*}

\noindent Applying inequality (\ref{ineq:errlast1}), using the previous bounds
for $\eta_{j}^{n}$ terms, and absorbing constants into a new constant $C$, completes the proof of Theorem  \ref{th:errBEFE-Ensemble}.
}
\end{document}